\documentclass[english,oneside,a4paper,11pt]{article}
\usepackage{amssymb}
\usepackage{latexsym}
\usepackage[english]{babel}
\usepackage{amsmath}
\usepackage{amsfonts}
\usepackage{amsbsy}
\usepackage[mathscr]{eucal}
\usepackage[latin1]{inputenc}
\usepackage{verbatim}
\usepackage{mathrsfs}
\usepackage{graphicx}
\usepackage{float}
\usepackage{bbm}
\usepackage{lscape}
\usepackage{mathdots}
\usepackage{empheq}
\usepackage[dvipsnames]{xcolor}
\usepackage{array}
\usepackage{a4wide}
\usepackage{enumitem}
\usepackage[dvips,colorlinks,bookmarks,breaklinks]{hyperref}  
\hypersetup{linkcolor=black,citecolor=black,filecolor=black,urlcolor=black}
\usepackage{slashbox,pict2e}
\usepackage{natbib}

\providecommand{\abs}[1]{\left\lvert#1 \right\rvert}
\providecommand{\norm}[1]{\left\lvert \! \left\lvert#1
\right\rvert \! \right\rvert}

\providecommand{\Log}[1]{\textnormal{Log} #1}

\newtheorem{theorem}{Theorem}

\newenvironment{proof}
{\begin{trivlist}\item[\hskip%
\labelsep{{\it \noindent Proof.}}]}{\hfill $\square$
\end{trivlist}}

\newcounter{counter}
\newcommand{\counter}{\stepcounter{counter}\thecounter}

\newcounter{examplecounter}
\newcommand{\examplecounter}{\stepcounter{examplecounter}\theexamplecounter}

\newenvironment{remark}
{\begin{trivlist}\item[\hskip%
\labelsep{{\it \noindent Remark \counter}}]}{\hfill
\end{trivlist}}

\newenvironment{example}
{\begin{trivlist}\item[\hskip%
\labelsep{{\sc \noindent Example \examplecounter}}]}{\hfill
\end{trivlist}}

\newenvironment{acknowledgements}
{\begin{trivlist}\item[\hskip%
\labelsep{\bf \noindent Acknowledgements.}]}{\hfill
\end{trivlist}}

\numberwithin{equation}{section}

\allowdisplaybreaks

\begin{document}
\begin{center}
{\LARGE \textbf{Almost sure convergence for weighted \\ sums of pairwise PQD random variables}} \\[25pt]
{\Large Jo\~{a}o Lita da Silva\footnote{\textit{E-mail address:} \texttt{jfls@fct.unl.pt; joao.lita@gmail.com}}} \\
\vspace{0.1cm}
\textit{Department of Mathematics and GeoBioTec \\ Faculty of Sciences and Technology \\
NOVA University of Lisbon \\ Quinta da Torre, 2829-516 Caparica,
Portugal}
\end{center}


\bigskip

\begin{abstract}
We obtain strong laws of large numbers of  Marcinkiewicz-Zygmund's type for weighted sums of pairwise positively quadrant dependent random variables stochastically dominated by a random variable $X \in \mathscr{L}_{p}$, $1 \leqslant p < 2$. We use our results to establish the strong consistency of estimators which emerge from regression models having pairwise positively quadrant dependent errors.
\end{abstract}

\bigskip


\bigskip


\bigskip

\section{Introduction}

\indent

The classical Kolmogorov's strong law of large numbers is, perhaps, the most famous strong limit theorem in Probability Theory. Originally presented in 1933 by Andrei Nikolaevich Kolmogorov (see \citealt{Kolmogorov33}), it can be stated as follows: if $X_{1},X_{2},\ldots$ is a sequence of independent and identically distributed random variables such that $\mathbb{E} \, X_{1}$ exists then
\begin{equation*}
\frac{X_{1} + X_{2} + \ldots + X_{n}}{n} \overset{\textnormal{a.s.}}{\longrightarrow} \mathbb{E} \, X_{1}.
\end{equation*}
In \citealt{Walk05}, an elegant short proof for a Kolmogorov's strong law of large numbers under very general assumptions was given by means of Tauberian theorems. In this paper, we shall follow Walk's statement to establish a Kolmogorov's strong law of large numbers for weighted pairwise positively quadrant dependent random variables. This result improves the corresponding one announced in \citealt{Louhichi00}  for sequences of (positively) associated random variables. Further, we shall present strong law of large numbers of Marcinkiewicz-Zygmund's type for weighted pairwise positively quadrant dependent random variables by using a maximal inequality recently announced in \citealt{Lita18} for these dependent structures.

It is important to stress out that in mathematical statistics, statistical physics or reliability theory, many stochastic models involving dependent random variables have arose over the last decades (see \citealt{Bulinski07}, \citealt{Hutchinson90}). For instance, most bivariate distributions in reliability theory are positively quadrant dependent (see \citealt{Hutchinson90}), whence, any efforts to establish formal results for these (and other) dependent structures are always of interest.

In last section, we provide some statistical applications of our assertions, namely, in strong consistency of estimators which can be found in regression models. Our statements not only extend others established lately in this issue (see \citealt{Lita13} and \citealt{Lita14}), but also allow us to consider statistical models having dependent random variables making them more appropriate and realistic.

Throughout, $x \wedge y$ and $x \vee y$ will stand for $\min\{x, y \}$ and $\max\{x,y \}$, respectively. Associated to a probability space $(\Omega, \mathcal{F}, \mathbb{P})$, we shall consider the space $\mathscr{L}_{p}$ $(p>0)$ of all measurable functions $X$ (necessarily random variables) for which $\mathbb{E} \lvert X \rvert^{p} < \infty$. For any measurable function $X$ we will define its positive and negative parts by $X^{+} = X \vee 0$ and $X^{-} = (-X) \vee 0$, respectively. Given an event $A$ we shall denote the indicator random variable of the event $A$ by $I_{A}$. All over, the function $x \mapsto \log (\abs{x} \vee \mathrm{e})$ will be denoted by $\Log \, x$. To make the computations be simpler looking, we shall employ the letter $C$ to denote any positive constant that can be explicitly computed, which is not necessarily the same on each appearance; the symbol $C(p)$
has identical meaning with the additional information that the constant depends on $p$.

\section{Almost sure convergence}

\indent

In the past, many authors have considered limit theorems involving pairwise positively quadrant dependent random variables (see, \citealt{Birkel89}, \citealt{Birkel93}, \citealt{Matula05} or \citealt{Newman84} among others). We shall proceed their study by establishing strong limits for weighted sums of these dependent random variables.

The concept of positively quadrant dependence for a pair of random variables due to \citealt{Lehmann66} can be given for sequences of random variables as follows: a sequence $\{X_{n}, \, n \geqslant 1 \}$ of random variables is said to be \emph{pairwise positively quadrant dependent} (pairwise PQD) if
\begin{equation*}
\mathbb{P} \left\{X_{k} \leqslant x_{k}, X_{j} \leqslant x_{j}  \right\} - \mathbb{P} \left\{X_{k} \leqslant x_{k} \right\} \mathbb{P} \left\{X_{j} \leqslant x_{j}  \right\} \geqslant 0
\end{equation*}
for all reals $x_{k}, x_{j}$ and all positive integers $k,j$ such that $k \neq j$.

Let $X,Y$ be random variables and $\ell$ a positive constant. Throughout, we shall consider the function $g_{\ell}(t) := (t \wedge \ell) \vee (- \ell)$ and the covariance quantity
\begin{equation}\label{eq:2.1}
G_{X,Y}(t) := \mathrm{Cov} \big(g_{t}(X),g_{t}(Y) \big) = \int_{-t}^{t} \int_{-t}^{t} \Delta_{X,Y}(x,y) \, \mathrm{d}x \, \mathrm{d}y
\end{equation}
where $\Delta_{X,Y}(x,y) := \mathbb{P} \left\{X \leqslant x, Y \leqslant y \right\} - \mathbb{P} \left\{X \leqslant x \right\} \mathbb{P} \left\{Y \leqslant y \right\}$. Further, a random sequence $\{X_{n}, \, n \geqslant 1 \}$ is \emph{stochastically dominated} by a random variable $X$ if there exists a constant $C>0$ such that $\sup_{n \geqslant 1} \mathbb{P} \left\{\abs{X_{n}} > t \right\} \leqslant C \, \mathbb{P} \left\{\abs{X} > t  \right\}$ for all $t > 0$ (see, for instance, \citealt{Lita15}).

Now, we state and prove a Kolmogorov's strong law of large numbers for weighted pairwise positively quadrant dependent random variables.

\begin{theorem}\label{thr:1}
Let $\{X_{n}, \, n \geqslant 1 \}$ be a sequence of pairwise PQD random variables stochastically
dominated by a random variable $X \in \mathscr{L}_{1}$. If $\{a_{n} \}$ is a sequence of constants satisfying $\sup_{n \geqslant 1} n^{-1} \sum_{k=1}^{n} a_{k}^{2} < \infty$ and
\begin{equation}\label{eq:2.2}
\sum_{1 \leqslant k < j \leqslant \infty} \lvert a_{k} a_{j} \rvert \int_{j}^{\infty} t^{-3} G_{X_{k},X_{j}}(t) \, \mathrm{d}t < \infty
\end{equation}
then $\sum_{k = 1}^{n} a_{k} (X_{k} - \mathbb{E} \, X_{k})/n \overset{\textnormal{a.s.}}{\longrightarrow} 0$.
\end{theorem}

\begin{proof}
By writing $a_{n} = a_{n}^{+} - a_{n}^{-}$, we may suppose without loss of generality that $a_{n}$ is non-negative for each $n$. Setting
\begin{equation}\label{eq:2.3}
\begin{gathered}
X_{n}' := g_{n}(X_{n}), \\
X_{n}'' := X_{n} - g_{n}(X_{n}), \\
Y_{n}' := (n \wedge X_{n}) \vee 0, \\
Y_{n}'' := (-n \vee X_{n}) \wedge 0
\end{gathered}
\end{equation}
we have $X_{n}' = Y_{n}' + Y_{n}''$ and $X_{n} = X_{n}' + X_{n}''$. From Lemma 1 of \citealt{Lehmann66} we get
\begin{gather*}
\mathrm{Cov}(Y_{k}',Y_{j}'') \geqslant 0, \\
\mathrm{Cov}(Y_{k}'',Y_{j}') \geqslant 0, \\
\mathrm{Cov}(Y_{k}'',Y_{j}'') \geqslant0
\end{gather*}
because $t \mapsto (n \wedge t) \vee 0$ and $t \mapsto (-n \vee t) \wedge 0$ are nondecreasing functions. Thus,
\begin{align*}
&\sum_{k,j=1}^{n} \mathrm{Cov}(a_{k} Y_{k}',a_{j} Y_{j}') = \\
&\qquad = \sum_{k,j=1}^{n} a_{k} a_{j} \mathrm{Cov}(Y_{k}',Y_{j}') \\
&\qquad \leqslant \sum_{k,j=1}^{n} a_{k} a_{j} \left[\mathrm{Cov}(Y_{k}',Y_{j}') + \mathrm{Cov}(Y_{k}',Y_{j}'') + \mathrm{Cov}(Y_{k}'',Y_{j}') + \mathrm{Cov}(Y_{k}'',Y_{j}'') \right] \\
&\qquad = \sum_{k,j=1}^{n} a_{k} a_{j} \mathrm{Cov}(X_{k}',X_{j}') \\
&\qquad = \sum_{k=1}^{n} a_{k}^{2} \mathbb{V}(X_{k}') + 2 \sum_{1 \leqslant k < j \leqslant n} a_{k} a_{j} \mathrm{Cov}(X_{k}',X_{j}').
\end{align*}
According to Abel's identity, we have
\begin{equation*}
\sum_{k=j}^{\infty} \frac{a_{k}^{2}}{k^{2}} \leqslant \frac{2}{j} \cdot \sup_{k \geqslant 1} \frac{1}{k} \sum_{m=1}^{k} a_{m}^{2} \leqslant \frac{C}{j}
\end{equation*}
for any $j \geqslant 1$, so that Lemma 1 of \citealt{Lita15} implies
\begin{align*}
\sum_{n=1}^{\infty} \frac{1}{n^{3}} \sum_{k=1}^{n} a_{k}^{2} \mathbb{V}(X_{k}') &= \sum_{k=1}^{\infty} \sum_{n=k}^{\infty} \frac{a_{k}^{2} \mathbb{V}(X_{k}')}{n^{3}} \\
&\leqslant C \sum_{k=1}^{\infty} \frac{a_{k}^{2} \mathbb{V}(X_{k}')}{k^{2}} \\
&\leqslant C \sum_{k=1}^{\infty} \frac{a_{k}^{2} \mathbb{E} \lvert X_{k}' \rvert^{2}}{k^{2}} \\
&\leqslant C \sum_{k=1}^{\infty} \frac{a_{k}^{2} \left(\mathbb{E} \, X_{k}^{2} I_{\left\{\lvert X_{k} \rvert \leqslant k \right\}} + k^{2} \mathbb{P} \left\{\lvert X_{k} \rvert > k \right\} \right)}{k^{2}} \\
&\leqslant C \sum_{k=1}^{\infty} \frac{a_{k}^{2} \left(\mathbb{E} \, X^{2} I_{\left\{\lvert X \rvert \leqslant k \right\}} + k^{2} \mathbb{P} \left\{\lvert X \rvert > k \right\} \right)}{k^{2}} \\
&= C \sum_{k=1}^{\infty} \frac{a_{k}^{2}}{k^{2}} \int_{0}^{k} u \, \mathbb{P} \left\{\abs{X} > u \right\} \mathrm{d}u \\
&= C \int_{0}^{\infty} u \, \mathbb{P} \left\{\abs{X} > u \right\} \sum_{\left\{k \colon k > u \right\}} \frac{a_{k}^{2}}{k^{2}} \, \mathrm{d}u \\
&\leqslant C \int_{0}^{\infty} \mathbb{P} \left\{\abs{X} > u \right\} \mathrm{d}u \\
&= C \, \mathbb{E} \abs{X} < \infty.
\end{align*}
Since
\begin{equation}\label{eq:2.4}
\begin{split}
& \sum_{n=1}^{\infty} \frac{1}{n^{3}} \sum_{1 \leqslant k < j \leqslant n} a_{k} a_{j} \mathrm{Cov}(X_{k}',X_{j}')
= \\
&\qquad =\sum_{n=1}^{\infty} \frac{1}{n^{3}} \sum_{1 \leqslant k < j \leqslant n} a_{k} a_{j} \int_{-k}^{k} \int_{-j}^{j} \Delta_{X_{k},X_{j}}(x,y) \, \mathrm{d}x \mathrm{d}y \\
&\qquad \leqslant \sum_{n=1}^{\infty} \sum_{1 \leqslant k < j \leqslant n} \frac{a_{k} a_{j} G_{X_{k},X_{j}}(n)}{n^{3}} \\
&\qquad= \sum_{1 \leqslant k < j \leqslant \infty} a_{k} a_{j} \int_{-\infty}^{\infty} \int_{-\infty}^{\infty} \sum_{n=1}^{\infty} \frac{I_{\left\{n \geqslant \abs{x} \vee \abs{y} \vee j \right\}}}{n^{3}} \Delta_{X_{k},X_{j}}(x,y) \, \mathrm{d} x \mathrm{d} y \\
&\qquad \leqslant C \sum_{1 \leqslant k < j \leqslant \infty} a_{k} a_{j} \int_{-\infty}^{\infty} \int_{-\infty}^{\infty} \frac{\Delta_{X_{k},X_{j}}(x,y)}{\left(\abs{x} \vee \abs{y} \vee j \right)^{2}} \, \mathrm{d} x \mathrm{d} y \\
&\qquad = C \sum_{1 \leqslant k < j \leqslant \infty} a_{k} a_{j} \int_{j}^{\infty} t^{-3} G_{X_{k},X_{j}}(t) \, \mathrm{d} t < \infty
\end{split}
\end{equation}
by Lemma 4 of \citealt{Louhichi00}, we obtain
\begin{equation*}
\sum_{n=1}^{\infty} \frac{1}{n^{3}} \sum_{k,j=1}^{n} \mathrm{Cov}(a_{k} Y_{k}',a_{j} Y_{j}') \leqslant \sum_{n=1}^{\infty} \frac{a_{k}^{2}}{n^{3}} \sum_{k=1}^{n} \mathbb{V}(X_{k}') + 2 \sum_{n=1}^{\infty} \frac{1}{n^{3}} \sum_{1 \leqslant k < j \leqslant n} a_{k} a_{j} \mathrm{Cov}(X_{k}',X_{j}') < \infty.
\end{equation*}
On the other hand, $a_{n} Y_{n}' \geqslant 0$ and Lemma 1 of \citealt{Lita15} yields
\begin{align*}
\frac{1}{n} \sum_{k=1}^{n} \mathbb{E}(a_{k} Y_{k}') &\leqslant \frac{1}{n} \sum_{k=1}^{n} a_{k} \mathbb{E} \lvert X_{k}' \rvert \\
&\leqslant \frac{C}{n} \sum_{k=1}^{n} a_{k} \left(\mathbb{E} \lvert X \rvert I_{\left\{\lvert X \rvert \leqslant k \right\}} + k \mathbb{P} \left\{\lvert X \rvert > k \right\} \right) \\
&\leqslant \frac{C \, \mathbb{E} \lvert X \rvert}{n} \sum_{k=1}^{n} a_{k} \\
&\leqslant C \, \mathbb{E} \lvert X \rvert \left(\frac{1}{n} \sum_{k=1}^{n} a_{k}^{2} \right)^{1/2} \\
&\leqslant C \, \mathbb{E} \lvert X \rvert < \infty
\end{align*}
which ensures
\begin{equation}\label{eq:2.5}
\frac{1}{n} \sum_{k=1}^{n} a_{k} (Y_{k}' - \mathbb{E} \, Y_{k}') \overset{\textnormal{a.s.}}{\longrightarrow} 0
\end{equation}
(see, for instance, Remark 3 of \citealt{Walk05}). Noting that $a_{n} Y_{n}'' \leqslant 0$, $\sum_{k=1}^{n} \mathbb{E}(-a_{k} Y_{k}'')/n \leqslant \sum_{k=1}^{n} a_{k} \mathbb{E} \lvert X_{k}' \rvert /n < \infty$ and

\begin{align*}
\sum_{k,j=1}^{n} \mathrm{Cov}(- a_{k} Y_{k}'', - a_{k} Y_{j}'') = \sum_{k,j=1}^{n} \mathrm{Cov}(a_{k} Y_{k}'', a_{k} Y_{j}'') \leqslant \sum_{k,j=1}^{n} a_{k} a_{j} \mathrm{Cov}(X_{k}',X_{j}')
\end{align*}
one can argue as above to conclude that
\begin{equation}\label{eq:2.6}
\frac{1}{n} \sum_{k=1}^{n} a_{k} \left(\mathbb{E} \, Y_{k}'' - Y_{k}'' \right) \overset{\textnormal{a.s.}}{\longrightarrow} 0.
\end{equation}
Hence, \eqref{eq:2.5} and \eqref{eq:2.6} yield $\sum_{k=1}^{n} a_{k} (X_{k}' - \mathbb{E} \, X_{k}')/n \overset{\textnormal{a.s.}}{\longrightarrow} 0$. It remains to prove
\begin{equation}\label{eq:2.7}
\frac{1}{n} \sum_{k=1}^{n} a_{k} (X_{k}'' - \mathbb{E} \, X_{k}'') \overset{\textnormal{a.s.}}{\longrightarrow} 0.
\end{equation}
Since
\begin{equation*}
\sum_{n=1}^{\infty} \mathbb{P} \left\{X_{n} \neq X_{n}' \right\} = \sum_{n=1}^{\infty} \mathbb{P} \left\{\lvert X_{n} \rvert > n \right\} \leqslant C \sum_{n=1}^{\infty} \mathbb{P} \left\{\lvert X \rvert > n \right\} \leqslant C \, \mathbb{E} \lvert X \rvert < \infty
\end{equation*}
it follows that a.s. $X_{n} = X_{n}'$ for all but a finite number of values of $n$, entailing
\begin{equation*}
\frac{1}{n} \sum_{k = 1}^{n} a_{k} X_{k}'' \overset{\textnormal{a.s.}}{\longrightarrow} 0.
\end{equation*}
From the dominated convergence theorem we have $\mathbb{E} \lvert X \rvert I_{\left\{\lvert X \rvert > n \right\}} = o(1)$ as $n \rightarrow \infty$, implying
\begin{align*}
\limsup_{n \rightarrow \infty} \abs{\frac{1}{n} \sum_{k = 1}^{n} \mathbb{E} (a_{k} X_{k}'')} &\leqslant \limsup_{n \rightarrow \infty} \frac{1}{n} \sum_{k = 1}^{n} a_{k} \mathbb{E} \lvert X_{k}'' \rvert \\
&\leqslant \limsup_{n \rightarrow \infty} \frac{1}{n} \sum_{k = 1}^{n} a_{k} \mathbb{E} \lvert X_{k} \rvert I_{\left\{\lvert X_{k} \rvert > k \right\}} \\
&\leqslant \limsup_{n \rightarrow \infty} \left(\frac{1}{n} \sum_{k=1}^{n} a_{k}^{2} \right)^{1/2} \left(\frac{1}{n} \sum_{k=1}^{n} \mathbb{E}^{2} \lvert X_{k} \rvert I_{\left\{\lvert X_{k} \rvert > k \right\}} \right)^{1/2} \\
&\leqslant \limsup_{n \rightarrow \infty} C \left(\frac{1}{n} \sum_{k=1}^{n} \mathbb{E}^{2} \lvert X \rvert I_{\left\{\lvert X \rvert > k \right\}} \right)^{1/2} \\
&= 0
\end{align*}
and \eqref{eq:2.7} holds establishing the thesis.
\end{proof}

\begin{remark}
We observe that condition \eqref{eq:2.2} can be replaced by the weaker condition
\begin{equation}\label{eq:2.8}
\sum_{1 \leqslant k < j \leqslant \infty} \frac{\lvert a_{k} a_{j} \rvert}{j^{2}} \int_{-k}^{k} \int_{-j}^{j} \Delta_{X_{k},X_{j}}(x,y) \, \mathrm{d}x \mathrm{d}y < \infty
\end{equation}
by waiving Lemma 4 of \citealt{Louhichi00} in the upper bound \eqref{eq:2.4}. When $a_{n} = 1$ for all $n$, Theorem~\ref{thr:1} equipped with \eqref{eq:2.8} instead of \eqref{eq:2.2} extends Corollary 1 of \citealt{Lita18} to $p = 1$ and it corresponds to Theorem 3 of \citealt{Chen19}.
\end{remark}

Our Theorem~\ref{thr:1} extends also Theorem 1 of \citealt{Louhichi00} to sequences of pairwise PQD random variables when $p = 1$. Recall that positively quadrant dependent random variables are not necessarily (positively) associated (see, for instance, \citealt{Esary67} or \citealt{Tong80}). Furthermore, the normalising constants in Theorem~\ref{thr:1} improve the considered ones in Theorem 2 of \citealt{Lita18} for $p = 1$. It is worthy to note that for the special case $p = 1$, the previous approach leads to sharped results discarding the direct use of any maximal inequality which is, in fact, the key ingredient in both \citealt{Louhichi00} and \citealt{Lita18}. In particular, Theorem~\ref{thr:1} with weights $a_{n} = 1$ for all $n$ does not require the finiteness of the variance in each random variable unlike Theorem 1 of \citealt{Birkel89}.

The statement below gives us the almost sure convergence for weighted sums of pairwise PQD random variables when the moment condition of the random variable $X$ in Theorem~\ref{thr:1} is strengthened.

\begin{theorem}\label{thr:2}
Let $1 < p < 2$ and $\{X_{n}, \, n \geqslant 1 \}$ be a sequence of pairwise PQD random variables stochastically dominated by a random variable $X \in \mathscr{L}_{p}$. If $\{a_{n} \}$ is a sequence of constants satisfying $\sup_{n \geqslant 1} n^{-1} \sum_{k=1}^{n} a_{k}^{2} < \infty$ and
\begin{equation}\label{eq:2.9}
\sum_{1 \leqslant k < j \leqslant \infty} \lvert a_{k} a_{j} \rvert \int_{\frac{j^{1/p}}{\Log^{2/p} j}}^{\infty} \frac{G_{X_{k},X_{j}}(t)}{t^{3} \Log^{2} \, t} \, \mathrm{d}t < \infty,
\end{equation}
then $\sum_{k = 1}^{n} a_{k}(X_{k} - \mathbb{E} \, X_{k})/(n^{1/p} \Log^{2(p - 1)/p} \, n) \overset{\textnormal{a.s.}}{\longrightarrow} 0$.
\end{theorem}

\begin{proof}
As in the proof of Theorem~\ref{thr:1}, we shall assume $a_{n} \geqslant 0$ for all $n$. Considering $X_{n}' := g_{n^{1/p}/\Log^{2/p}n}(X_{n})$ and $X_{n}'' := X_{n} - g_{n^{1/p}/ \Log^{2/p}n}(X_{n})$ it follows that $\{a_{n} X_{n}', \, n \geqslant 1 \}$ is a sequence of pairwise PQD random variables. From Lemma 1 of \citealt{Lita18} we obtain, for each $\varepsilon > 0$ and a fixed $n_{0}$,
\begin{align*}
& \sum_{n=n_{0}}^{\infty} \frac{1}{n} \mathbb{P} \left\{\max_{1 \leqslant k \leqslant n} \abs{\sum_{j=1}^{k} a_{j}(X_{j}' - \mathbb{E} \, X_{j}')} > \varepsilon  n^{1/p} \Log^{2(p - 1)/p} \, n \right\} \leqslant \\
&\leqslant \frac{C}{\varepsilon^{2}} \sum_{n=n_{0}}^{\infty} \sum_{j=1}^{n} \frac{\Log^{2(2 - p)/p}n}{n^{1 + 2/p}} \mathrm{Cov}\left[\sum_{i=1}^{j} a_{i} (X_{i}' - \mathbb{E} \, X_{i}'),a_{j}(X_{j}' - \mathbb{E} \, X_{j}') \right] \\
&\leqslant \frac{C}{\varepsilon^{2}} \sum_{n=n_{0}}^{\infty} \sum_{j=1}^{n} \frac{\Log^{2(2 - p)/p}n \, \mathbb{E} (a_{j} X_{j}')^{2}}{n^{1 + 2/p}} + \frac{C}{\varepsilon^{2}} \sum_{n=n_{0}}^{\infty} \sum_{1 \leqslant i < j \leqslant n} \frac{\Log^{2(2 - p)/p}n \, \mathrm{Cov}(a_{i} X_{i}',a_{j} X_{j}')}{n^{1 + 2/p}} \\
&\leqslant \frac{C}{\varepsilon^{2}} \sum_{n=n_{0}}^{\infty} \sum_{j=1}^{n} \frac{a_{j}^{2} \, \Log^{2(2 - p)/p}n}{n^{1 + 2/p}} \left(\mathbb{E} \, X_{j}^{2} I_{\left\{\lvert X_{j} \rvert \leqslant \frac{j^{1/p}}{\Log^{2/p} j} \right\}} + \frac{j^{2/p}}{\Log^{4/p} \, j} \mathbb{P} \left\{\lvert X_{j} \rvert > \frac{j^{1/p}}{\Log^{2/p} j} \right\} \right) + \\
& \qquad \frac{C}{\varepsilon^{2}} \sum_{n=n_{0}}^{\infty} \sum_{1 \leqslant i < j \leqslant n} \frac{a_{i} a_{j} \, \Log^{2(2 - p)/p}n}{n^{1 + 2/p}} \int_{-i^{1/p}/\Log^{2/p} i}^{i^{1/p}/\Log^{2/p} i} \int_{-j^{1/p}/\Log^{2/p} j}^{j^{1/p}/\Log^{2/p} j} \Delta_{X_{i},X_{j}}(x,y) \, \mathrm{d}x \mathrm{d}y \\
&\leqslant \frac{C}{\varepsilon^{2}} \sum_{n=n_{0}}^{\infty} \sum_{j=1}^{n} \frac{a_{j}^{2} \, \Log^{2(2 - p)/p}n}{n^{1 + 2/p}} \left(\mathbb{E} \, X_{j}^{2} I_{\left\{\lvert X_{j} \rvert \leqslant \frac{n^{1/p}}{\Log^{2/p} n} \right\}} + \frac{n^{2/p}}{\Log^{4/p} \, n} \mathbb{P} \left\{\lvert X_{j} \rvert > \frac{n^{1/p}}{\Log^{2/p} n} \right\} \right) + \\
& \qquad \frac{C}{\varepsilon^{2}} \sum_{n=n_{0}}^{\infty} \sum_{1 \leqslant i < j \leqslant n} \frac{a_{i} a_{j} \, \Log^{2(2 - p)/p}n}{n^{1 + 2/p}} \int_{-n^{1/p}/\Log^{2/p} n}^{n^{1/p}/\Log^{2/p} n} \int_{-n^{1/p}/\Log^{2/p} n}^{n^{1/p}/\Log^{2/p} n} \Delta_{X_{i},X_{j}}(x,y) \, \mathrm{d}x \mathrm{d}y \\
&\leqslant \frac{C}{\varepsilon^{2}} \sum_{n=1}^{\infty} \sum_{j=1}^{n} \frac{a_{j}^{2} \, \Log^{2(2 - p)/p}n}{n^{1 + 2/p}} \left(\mathbb{E} \, X^{2} I_{\left\{\abs{X} \leqslant \frac{n^{1/p}}{\Log^{2/p} n} \right\}} + \frac{n^{2/p}}{\Log^{4/p} \, n} \mathbb{P} \left\{\abs{X} > \frac{n^{1/p}}{\Log^{2/p} n} \right\} \right) + \\
& \qquad \frac{C}{\varepsilon^{2}} \sum_{n=1}^{\infty} \sum_{1 \leqslant i < j \leqslant n} \frac{a_{i} a_{j} \Log^{2(2 - p)/p}n}{n^{1 + 2/p}} G_{X_{i},X_{j}} \left(\frac{n^{1/p}}{\Log^{2/p} n} \right) \\
&\leqslant \frac{C}{\varepsilon^{2}} \sum_{n=1}^{\infty} \frac{\Log^{2(2 - p)/p}n}{n^{2/p}} \mathbb{E} \, X^{2} I_{\left\{\abs{X} \leqslant \frac{n^{1/p}}{\Log^{2/p} n} \right\}} + \frac{C}{\varepsilon^{2}} \sum_{n=1}^{\infty} \frac{1}{\Log^{2} \, n} \mathbb{P} \left\{\abs{X} > \frac{n^{1/p}}{\Log^{2/p} n} \right\} + \\
& \qquad \frac{C}{\varepsilon^{2}} \sum_{n=1}^{\infty} \sum_{1 \leqslant i < j \leqslant n} \frac{a_{i} a_{j} \, \Log^{2(2 - p)/p}n}{n^{1 + 2/p}} G_{X_{i},X_{j}} \left(\frac{n^{1/p}}{\Log^{2/p} n} \right)
\end{align*}
since $\sup_{n \geqslant 1} \sum_{j=1}^{n} a_{j}^{2}/n < \infty$. Supposing
\begin{equation}\label{eq:2.10}
A_{j} = \left\{(j - 1)^{1/p}/\Log^{2/p} (j - 1) < \abs{X} \leqslant j^{1/p}/\Log^{2/p} j \right\}, \quad j \geqslant 1
\end{equation}
we have
\begin{align*}
\sum_{n=1}^{\infty} \frac{\Log^{2(2 - p)/p} n}{n^{2/p}} \, \mathbb{E} \, X^{2} I_{\left\{\lvert X \rvert \leqslant \frac{n^{1/p}}{\Log^{2/p} n} \right\}} &= \sum_{n=1}^{\infty} \sum_{j=1}^{n} \frac{\Log^{2(2 - p)/p} n}{n^{2/p}} \, \mathbb{E} \, X^{2} I_{A_{j}} \\
& = \sum_{j=1}^{\infty} \sum_{n=j}^{\infty} \frac{\Log^{2(2 - p)/p} n}{n^{2/p}} \, \mathbb{E} \, X^{2} I_{A_{j}} \\
& \leqslant C \sum_{j=1}^{\infty} \frac{\Log^{2(2 - p)/p} j}{j^{2/p - 1}} \, \mathbb{E} \, X^{2} I_{A_{j}} \\
& \leqslant C \sum_{j=1}^{\infty} \, \mathbb{E} \abs{X}^{p} I_{A_{j}} \\
& = C \, \mathbb{E} \abs{X}^{p} < \infty
\end{align*}
and
\begin{equation*}
\sum_{n=1}^{\infty} \frac{1}{\Log^{2} n} \mathbb{P} \left\{\abs{X} > \frac{n^{1/p}}{\Log^{2/p} n}  \right\} = \sum_{n=1}^{\infty} \frac{1}{\Log^{2} n} \, \mathbb{P} \left\{\abs{X}^{p} > \frac{n}{\Log^{2} n} \right\} \leqslant C \, \mathbb{E} \abs{X}^{p} < \infty.
\end{equation*}
In order to prove
\begin{equation}\label{eq:2.11}
\sum_{n=1}^{\infty} \sum_{1 \leqslant i < j \leqslant n} \frac{a_{i} a_{j} \, \Log^{2(2 - p)/p}n}{n^{1 + 2/p}} G_{X_{i},X_{j}} \left(\frac{n^{1/p}}{\Log^{2/p} n} \right) < \infty
\end{equation}
we have
\begin{equation}\label{eq:2.12}
\begin{split}
& \sum_{n=1}^{\infty} \sum_{1 \leqslant i < j \leqslant n} \frac{a_{i} a_{j} \, \Log^{2(2 - p)/p}n}{n^{1 + 2/p}} G_{X_{i},X_{j}} \left(\frac{n^{1/p}}{\Log^{2/p} n} \right) = \\
&\sum_{1 \leqslant i < j \leqslant \infty} a_{i} a_{j} \int_{-\infty}^{\infty} \int_{-\infty}^{\infty} \sum_{n=1}^{\infty} \frac{\Log^{2(2 - p)/p}n}{n^{1 + 2/p}} \, I_{\left\{\frac{n^{1/p}}{\Log^{2/p} \, n} \geqslant \abs{x} \right\}} \, I_{\left\{\frac{n^{1/p}}{\Log^{2/p} \, n} \geqslant \abs{y} \right\}} I_{\left\{n \geqslant j \right\}} \Delta_{X_{i},X_{j}}(x,y) \, \mathrm{d}x \mathrm{d}y.
\end{split}
\end{equation}
Since $p^{2} t^{p} \Log^{2} \, t$ is an asymptotic inverse of $t^{1/p}/\Log^{2/p} t$ (see \citealt{Bingham87}, page $28$), we get
\begin{equation}\label{eq:2.13}
\begin{split}
& \sum_{n=1}^{\infty} \frac{\Log^{2(2 - p)/p}n}{n^{1 + 2/p}} \, I_{\left\{\frac{n^{1/p}}{\Log^{2/p} \, n} \geqslant \abs{x} \right\}} \, I_{\left\{\frac{n^{1/p}}{\Log^{2/p} \, n} \geqslant \abs{y} \right\}} I_{\left\{n \geqslant j \right\}} \leqslant \\
&\qquad \leqslant \sum_{n=1}^{\infty} \frac{\Log^{2(2 - p)/p}n}{n^{1 + 2/p}} \, I_{\left\{C(p) n \geqslant p^{2} \abs{x}^{p} \Log^{2} \abs{x} \right\}} \, I_{\left\{C(p) n \geqslant p^{2} \abs{y}^{p} \Log^{2} \abs{y} \right\}} I_{\left\{n \geqslant j \right\}} \\
&\qquad = \sum_{n=1}^{\infty} \frac{\Log^{2(2 - p)/p}n}{n^{1 + 2/p}} \, I_{\left\{n \geqslant \frac{p^{2} \abs{x}^{p} \Log^{2} \abs{x}}{C(p)} \vee \frac{p^{2} \abs{y}^{p} \Log^{2} \abs{y}}{C(p)} \vee j \right\}} \\
&\qquad \leqslant C(p) \cdot \frac{\Log^{2(2 - p)/p} \left[\frac{p^{2} \abs{x}^{p} \Log^{2} \abs{x}}{C(p)} \vee \frac{p^{2} \abs{y}^{p} \Log^{2} \abs{y}}{C(p)} \vee j \right]}{\left[\frac{p^{2} \abs{x}^{p} \Log^{2} \abs{x}}{C(p)} \vee \frac{p^{2} \abs{y}^{p} \Log^{2} \abs{y}}{C(p)} \vee j \right]^{2/p}}.
\end{split}
\end{equation}
Putting
\begin{gather*}
m = \sup_{j \geqslant 1} \frac{\Log^{2(2 - p)/p} j}{j^{2/p}}, \\
u(t) = \frac{\Log^{2(2 - p)/p} \left[\frac{p^{2} t^{p} \Log^{2} \, t}{C(p)} \right]}{t^{2} \Log^{4/p} t}
\end{gather*}
it follows
\begin{gather*}
u(t) \sim \frac{p^{2(2 - p)/p}}{t^{2} \Log^{2} \, t}, \quad t \rightarrow \infty, \\
u(t) \sim \frac{1}{t^{2} \Log^{2} \, t}, \quad t \rightarrow 0^{+}
\end{gather*}
and there is a constant $M(p) > 1$ such that
\begin{equation*}
\sup_{t > 0} \frac{u(t)}{\frac{1}{t^{2} \Log^{2} t}} \leqslant M(p) < \infty.
\end{equation*}
Hence, for all $x \neq 0$ and $y \neq 0$,
\begin{equation}\label{eq:2.14}
\begin{split}
&\frac{\Log^{2(2 - p)/p} \left[\frac{p^{2} \abs{x}^{p} \Log^{2} \abs{x}}{C(p)} \vee \frac{p^{2} \abs{y}^{p} \Log^{2} \abs{y}}{C(p)} \vee j \right]}{\left[\frac{p^{2} \abs{x}^{p} \Log^{2} \abs{x}}{C(p)} \vee \frac{p^{2} \abs{y}^{p} \Log^{2} \abs{y}}{C(p)} \vee j \right]^{2/p}} = \\
&\qquad = \int_{0}^{m} I_{\left\{t \leqslant u(\lvert x \rvert) \right\}} I_{\left\{t \leqslant u(\lvert y \rvert) \right\}} I_{\left\{t \leqslant \frac{\Log^{2(2 - p)/p}j}{j^{2/p}} \right\}} \, \mathrm{d}t \\
&\qquad \leqslant \int_{0}^{m} I_{\left\{t \leqslant \frac{M(p)}{x^{2} \Log^{2} \abs{x}} \right\}} I_{\left\{t \leqslant \frac{M(p)}{y^{2} \Log^{2} \abs{y}} \right\}} I_{\left\{t \leqslant \frac{\Log^{2(2 - p)/p}j}{j^{2/p}} \right\}} \, \mathrm{d}t \\
&\qquad = \int_{0}^{m} I_{\left\{\lvert x \rvert \leqslant v^{-1}(t) \right\}} I_{\left\{\lvert y \rvert \leqslant v^{-1}(t) \right\}} I_{\left\{t \leqslant \frac{\Log^{2(2 - p)/p}j}{j^{2/p}} \right\}} \, \mathrm{d}t
\end{split}
\end{equation}
where $v^{-1}(t)$ denotes the inverse of $v(t) = M(p)/(t^{2} \Log^{2} \, t)$, $t > 0$ and according to Fubini's theorem, we obtain
\begin{align}
&\int_{-\infty}^{\infty} \int_{-\infty}^{\infty} \sum_{n=1}^{\infty} \frac{\Log^{2(2 - p)/p}n}{n^{1 + 2/p}} \, I_{\left\{\frac{n^{1/p}}{\Log^{2/p} \, n} \geqslant \abs{x} \right\}} \, I_{\left\{\frac{n^{1/p}}{\Log^{2/p} \, n} \geqslant \abs{y} \right\}} I_{\left\{n \geqslant j \right\}} \Delta_{X_{i},X_{j}}(x,y) \, \mathrm{d}x \mathrm{d}y \notag \\
&\qquad \leqslant C(p) \int_{0}^{m} I_{\left\{t \leqslant \frac{\Log^{2(2 - p)/p}j}{j^{2/p}} \right\}} \, G_{X_{i},X_{j}}\left[v^{-1}(t) \right] \mathrm{d}t \notag \\
&\qquad = C(p) \int_{0}^{\frac{\Log^{2(2 - p)/p}j}{j^{2/p}}} G_{X_{i},X_{j}}\left[v^{-1}(t) \right] \mathrm{d}t \notag \\
&\qquad = - C(p) \int_{v^{-1}\left(\frac{\Log^{2(2 - p)/p}j}{j^{2/p}} \right)}^{\infty} v'(s) G_{X_{i},X_{j}}(s) \, \mathrm{d}s \notag \\
&\qquad \leqslant C(p) \int_{v^{-1}\left(\frac{\Log^{2(2 - p)/p}j}{j^{2/p}} \right)}^{\infty} \left[\frac{1}{s^{3} \Log^{2} \, s} + \frac{1}{(s \vee \mathrm{e}) s^{2} \Log^{3} \, s} \right]  G_{X_{i},X_{j}}(s) \, \mathrm{d}s \label{eq:2.15} \\
&\qquad = C(p) \int_{\frac{p \sqrt{M(p)} j^{1/p}}{\Log^{2/p} j} + o \left(\frac{j^{1/p}}{\Log^{2/p} j} \right)}^{\infty} \left[\frac{1}{s^{3} \Log^{2} \, s} + \frac{1}{(s \vee \mathrm{e}) s^{2} \Log^{3} \, s} \right]  G_{X_{i},X_{j}}(s) \, \mathrm{d}s \notag \\
&\qquad \leqslant C(p) \int_{\frac{p \sqrt{M(p)} j^{1/p}}{\Log^{2/p} j} + o \left(\frac{j^{1/p}}{\Log^{2/p} j} \right)}^{\infty} \frac{G_{X_{i},X_{j}}(s)}{s^{3} \Log^{2} \, s} \, \mathrm{d}s \notag \\
&\qquad \leqslant C(p) \int_{\frac{j^{1/p}}{\Log^{2/p} j}}^{\infty} \frac{G_{X_{i},X_{j}}(s)}{s^{3} \Log^{2} \, s} \, \mathrm{d}s \notag
\end{align}
for $j$ large enough, because
\begin{equation*}
v^{-1}(t) \sim \frac{\sqrt{M(p)}}{\sqrt{t} \abs{\log (\sqrt{t} \wedge \mathrm{e})}}, \quad t \rightarrow 0^{+}
\end{equation*}
and $s \mapsto G_{X_{i},X_{j}}(s)$ is a nonnegative and nondecreasing function. Thus, gathering \eqref{eq:2.12}, \eqref{eq:2.13}, \eqref{eq:2.14} and \eqref{eq:2.15} we obtain \eqref{eq:2.11} by using \eqref{eq:2.9}. Hence,
\begin{equation*}
\sum_{n=1}^{\infty} \frac{1}{n} \mathbb{P} \left\{\max_{1 \leqslant k \leqslant n} \abs{\sum_{j=1}^{k} a_{j}(X_{j}' - \mathbb{E} \, X_{j}')} > \varepsilon  n^{1/p} \Log^{2(p - 1)/p} \, n \right\} < \infty
\end{equation*}
and Theorem 2.1 of \citealt{Yang08} yields $\sum_{k = 1}^{n} a_{k}(X_{k}' - \mathbb{E} \, X_{k}')/(n^{1/p} \Log^{2(p - 1)/p} \, n) \overset{\textnormal{a.s.}}{\longrightarrow} 0$. It remains to show
\begin{equation}\label{eq:2.16}
\frac{1}{n^{1/p} \Log^{2(p - 1)/p} \, n} \sum_{k = 1}^{n} a_{k}(X_{k}'' - \mathbb{E} \, X_{k}'') \overset{\textnormal{a.s.}}{\longrightarrow} 0.
\end{equation}
By virtue of the Kronecker's lemma, convergence \eqref{eq:2.16} holds if
\begin{equation}\label{eq:2.17}
\sum_{k=1}^{\infty} \frac{1}{k^{1/p} \Log^{2(p - 1)/p} \, k} \lvert a_{k} (X_{k}'' - \mathbb{E} \, X_{k}'') \rvert < \infty \quad \text{a.s.}
\end{equation}
Since, for all $k \geqslant 3$,
\begin{align*}
&\frac{1}{k^{1/p} \Log^{2(p - 1)/p} \, k} - \frac{1}{(k + 1)^{1/p} \Log^{2(p - 1)/p} (k + 1)} \\
&\quad = \int_{k}^{k+1} \frac{\Log^{(2 - 2p)/p} x + 2(p - 1) \Log^{(2 - 3p)/p} x}{p x^{1 + 1/p}} \, \mathrm{d}x \\
&\quad \leqslant \frac{2p - 1}{p} \int_{k}^{k+1} \frac{\Log^{(2 - 2p)/p} x}{x^{1 + 1/p}} \, \mathrm{d}x \\
&\quad \leqslant \frac{2p - 1}{p} \cdot \frac{1}{k^{1 + 1/p}{\Log^{2(p - 1)/p} k}}
\end{align*}
it follows
\begin{align*}
&\sum_{k=3}^{\infty} \frac{1}{k^{1/p} \Log^{2(p - 1)/p} \, k} \mathbb{E} \lvert a_{k} (X_{k}'' - \mathbb{E} \, X_{k}'') \rvert \\
&\quad \leqslant 2 \sum_{k=3}^{\infty} \frac{1}{k^{1/p} \Log^{2(p - 1)/p} \, k} \mathbb{E} \lvert a_{k} X_{k}'' \rvert \\
&\quad \leqslant 2 \sum_{k=3}^{\infty} \frac{1}{k^{1/p} \Log^{2(p - 1)/p} \, k} \mathbb{E} \lvert a_{k} X_{k} \rvert I_{\left\{\lvert X_{k} \rvert > k^{1/p}/ \Log^{2/p} k \right\}} \\
&\quad \leqslant C \sum_{k=3}^{\infty} \frac{\lvert a_{k} \rvert}{k^{1/p} \Log^{2(p - 1)/p} \, k} \mathbb{E} \lvert X \rvert I_{\left\{\lvert X \rvert > k^{1/p}/ \Log^{2/p} k \right\}} \\
&\quad = C \sum_{k=3}^{\infty} \sum_{j=k}^{\infty} \frac{\lvert a_{k} \rvert}{k^{1/p} \Log^{2(p - 1)/p} \, k} \mathbb{E} \lvert X \rvert I_{A_{j+1}} \\
&\quad = C \sum_{j=3}^{\infty} \sum_{k=3}^{j} \frac{\lvert a_{k} \rvert}{k^{1/p} \Log^{2(p - 1)/p} \, k} \mathbb{E} \lvert X \rvert I_{A_{j+1}} \\
&\quad \leqslant C \sum_{j=3}^{\infty} \left[\sum_{k=3}^{j} \left(\frac{1}{k^{1/p} \Log^{2(p - 1)/p} \, k} - \frac{1}{(k + 1)^{1/p} \Log^{2(p - 1)/p} (k + 1)} \right) \sum_{m=1}^{k} \lvert a_{m} \rvert \right. \\
&\qquad + \left. \frac{1}{(j + 1)^{1/p} \Log^{2(p - 1)/p} (j + 1)} \sum_{m=1}^{j} \lvert a_{m} \rvert \right] \mathbb{E} \lvert X \rvert I_{A_{j+1}} \\
&\quad \leqslant C(p) \sum_{j=3}^{\infty} \left[\sum_{k=3}^{j} \frac{1}{k^{1 + 1/p} \Log^{2(p - 1)/p} k} \sum_{m=1}^{k} \lvert a_{m} \rvert \right. \\
&\qquad \left. + \frac{1}{(j + 1)^{1/p} \Log^{2(p - 1)/p} (j + 1)} \sum_{m=1}^{j} \lvert a_{m} \rvert \right] \mathbb{E} \lvert X \rvert I_{A_{j+1}} \\
&\quad \leqslant C(p) \sum_{j=1}^{\infty} \left[\sum_{k=1}^{j} \frac{1}{k^{1/p} \Log^{2(p - 1)/p} k} + \frac{1}{(j + 1)^{1/p - 1} \Log^{2(p - 1)/p} (j + 1)} \right] \mathbb{E} \lvert X \rvert I_{A_{j+1}} \\
&\quad \leqslant C(p) \sum_{j=1}^{\infty} \left[\frac{1}{j^{1/p - 1} \Log^{2(p - 1)/p} \, j} + \frac{1}{(j + 1)^{1/p - 1} \Log^{2(p - 1)/p} (j + 1)} \right] \mathbb{E} \lvert X \rvert I_{A_{j+1}} \\
&\quad \leqslant C(p) \sum_{j=1}^{\infty} \frac{1}{j^{1/p - 1} \Log^{2(p - 1)/p} \, j} \mathbb{E} \lvert X \rvert I_{A_{j+1}} \\
&\quad \leqslant C(p) \sum_{j=1}^{\infty} \frac{1}{j^{1/p - 1} \Log^{2(p - 1)/p} \, j} \cdot \frac{j^{(1 - p)/p}}{\Log^{2(1 - p)/p} j} \mathbb{E} \lvert X \rvert^{p} I_{A_{j+1}} \\
&\quad \leqslant C(p) \mathbb{E} \lvert X \rvert^{p} < \infty
\end{align*}
where $A_{j+1}$ is defined in \eqref{eq:2.10}. Therefore, a.s. convergence \eqref{eq:2.17} is assured and
\begin{equation*}
\frac{1}{n^{1/p} \Log^{2(p - 1)/p} \, n} \sum_{k=1}^{n} a_{k}(X_{k}'' - \mathbb{E} \, X_{k}'') \overset{\textnormal{a.s.}}{\longrightarrow} 0.
\end{equation*}
The proof is complete.
\end{proof}

\begin{remark}
For any $r \geqslant s > 0$, $\sum_{k=1}^{n} \lvert a_{k} \rvert^{s}/n \leqslant \left(\sum_{k=1}^{n} \lvert a_{k} \rvert^{r}/n \right)^{s/r}$ by H\"{o}lder's inequality which implies that assumption $\sup_{n \geqslant 1} \sum_{k=1}^{n} a_{k}^{2}/n < \infty$ in both Theorems~\ref{thr:1} and~\ref{thr:2} can be replaced by the (stronger) condition $\sup_{n \geqslant 1} \sum_{k=1}^{n} \lvert a_{k} \rvert^{q}/n < \infty$, $q \geqslant 2$.
\end{remark}

The example below provides us random variables satisfying \eqref{eq:2.2} or \eqref{eq:2.9} but not possessing finite second moments.

\begin{example}
Let $1 \leqslant p < 2$ and $\{X_{n}, \, n \geqslant 1 \}$ be a sequence of random variables such that for every $k \neq j$, $(X_{k},X_{j})$ has Farlie-Gumbel-Morgenstern bivariate distribution, i.e.
\begin{equation*}
F_{X_{k},X_{j}}(x,y) = F_{X_{k}}(x) F_{X_{j}}(y) + \rho F_{X_{k}}(x) F_{X_{j}}(y) \big[1 - F_{X_{k}}(x) \big] \big[1 - F_{X_{j}}(y) \big], \quad 0 \leqslant \rho \leqslant 1.
\end{equation*}
Thus, $\{X_{n}, \, n \geqslant 1 \}$ is a pairwise PQD sequence (see, for instance, \citealt{Lai00}). Supposing that $X_{n}$ has probability density function $f_{X_{n}}(t) = n(q - 1)(nt)^{-q} I_{(1/n,\infty)}(t)$, $p + 1 < q \leqslant 3$, it follows that $\{X_{n}, \, n \geqslant 1 \}$ is stochastically dominated by $X_{1}$ and $\mathbb{E} \, \lvert X_{1} \rvert^{p} = (q - 1)/(q - p - 1)$. Further, $\mathbb{E} \, X_{n}^{2} = \infty$ for all $n$ and after standard computations we obtain $G_{X_{k},X_{j}}(t) \leqslant \rho \, C(q)/(kj)$. Hence,
\begin{equation*}
\sum_{1 \leqslant k < j \leqslant \infty} \lvert a_{k} a_{j} \rvert \int_{j}^{\infty} t^{-3} G_{X_{k},X_{j}}(t) \, \mathrm{d}t \leqslant \rho \, C(q) \sum_{1 \leqslant k < j \leqslant \infty} \left(\frac{1}{j^{3}} + \frac{1}{k j^{2}} \right) < \infty
\end{equation*}
when $p = 1$, and
\begin{equation*}
\sum_{1 \leqslant k < j \leqslant \infty} \lvert a_{k} a_{j} \rvert \int_{\frac{j^{1/p}}{\Log^{2/p} j}}^{\infty} \frac{G_{X_{k},X_{j}}(t)}{t^{3} \Log^{2} \, t} \, \mathrm{d}t \leqslant \rho \, C(p,q) \sum_{1 \leqslant k < j \leqslant \infty} \Log^{(4 - 2p)/p} j \left(\frac{1}{j^{1 + 2/p}} + \frac{1}{k j^{2/p}} \right) < \infty
\end{equation*}
with $C(p,q)$ a positive constant depending only on $p$ and $q$ whenever $1 < p < 2$, by noting that $\sup_{n \geqslant 1} n^{-1} \sum_{k=1}^{n} a_{k}^{2} < \infty$ entails
\begin{equation*}
\lvert a_{k} a_{j} \rvert \leqslant \frac{a_{k}^{2} + a_{j}^{2}}{2} \leqslant \frac{\sum_{\ell=1}^{k} a_{\ell}^{2} + \sum_{\ell=1}^{j} a_{\ell}^{2}}{2} \leqslant C(k + j).
\end{equation*}
\end{example}

Notice that if the weights $\{a_{n} \}$ satisfy $a_{n} = 1$ for all $n$ and $\{X_{n}, \, n \geqslant 1 \}$ is a sequence of pairwise PQD random variables such that
\begin{equation}\label{eq:2.18}
\Delta_{X_{k},X_{j}}(x,y) = \Delta_{X_{1},X_{j}}(x,y)
\end{equation}
for any $1 \leqslant k < j$ and every $x,y \in \mathbb{R}$, then condition \eqref{eq:2.2} can be simplified to
\begin{equation}\label{eq:2.19}
\sum_{j=2}^{\infty} \int_{j}^{\infty} \frac{G_{X_{1},X_{j}}(v)}{v^{2}} \, \mathrm{d}v < \infty
\end{equation}
(see Remark (6) of \citealt{Louhichi00}) or even to the less restrictive assumption
\begin{equation*}
\sum_{j=2}^{\infty} \frac{G_{X_{1},X_{j}}(j)}{j} < \infty
\end{equation*} 
because
\begin{align*}
&\sum_{n=1}^{\infty} \frac{1}{n^{3}} \sum_{1 \leqslant k < j \leqslant n} \int_{-k}^{k} \int_{-j}^{j} \Delta_{X_{k},X_{j}}(x,y) \, \mathrm{d}x \mathrm{d}y = \\
&\qquad = \sum_{n=1}^{\infty} \frac{1}{n^{3}} \sum_{1 \leqslant k < j \leqslant n} \int_{-k}^{k} \int_{-j}^{j} \Delta_{X_{1},X_{j}}(x,y) \, \mathrm{d}x \mathrm{d}y \\
&\qquad \leqslant \sum_{n=2}^{\infty} \sum_{j=2}^{n} \frac{G_{X_{1},X_{j}}(n)}{n^{2}} \\
&\qquad = \sum_{j=2}^{\infty} \int_{-\infty}^{\infty} \int_{-\infty}^{\infty} \sum_{n=1}^{\infty} \frac{I_{\left\{n \geqslant \lvert x \rvert \vee \lvert y \rvert \vee j \right\}}}{n^{2}} \Delta_{X_{1},X_{j}}(x,y) \, \mathrm{d}x \mathrm{d}y \\
&\qquad \leqslant C \sum_{j=2}^{\infty} \int_{-\infty}^{\infty} \int_{-\infty}^{\infty} \frac{\Delta_{X_{1},X_{j}}(x,y)}{\lvert x \rvert \vee \lvert y \rvert \vee j} \, \mathrm{d}x \mathrm{d}y \\
&\qquad = C \sum_{j=2}^{\infty} \int_{-\infty}^{\infty} \int_{-\infty}^{\infty} \int_{0}^{1} I_{\left\{\lvert x \rvert \leqslant 1/u \right\}} I_{\left\{\lvert y \rvert \leqslant 1/u \right\}} I_{\left\{u \leqslant 1/j \right\}} \Delta_{X_{1},X_{j}}(x,y) \, \mathrm{d}u \, \mathrm{d}x \mathrm{d}y \\
&\qquad = C \sum_{j=2}^{\infty} \int_{0}^{1/j} G_{X_{1},X_{j}}\left(\frac{1}{u} \right) \, \mathrm{d}u \\
&\qquad = C \sum_{j=2}^{\infty} \int_{j}^{\infty} \frac{G_{X_{1},X_{j}}(v)}{v^{2}} \, \mathrm{d}v
\end{align*}
and
\begin{align*}
\sum_{n=1}^{\infty} \frac{1}{n^{3}} \sum_{1 \leqslant k < j \leqslant n} \int_{-k}^{k} \int_{-j}^{j} \Delta_{X_{k},X_{j}}(x,y) \, \mathrm{d}x \mathrm{d}y &= \sum_{n=1}^{\infty} \frac{1}{n^{3}} \sum_{1 \leqslant k < j \leqslant n} \int_{-k}^{k} \int_{-j}^{j} \Delta_{X_{1},X_{j}}(x,y) \, \mathrm{d}x \mathrm{d}y \\
&\leqslant \sum_{n=2}^{\infty} \sum_{j=2}^{n} \frac{(j - 1) G_{X_{1},X_{j}}(j)}{n^{3}} \\ &\leqslant C \sum_{j=2}^{\infty} \frac{G_{X_{1},X_{j}}(j)}{j}
\end{align*}
(recall that $t \mapsto G_{X_{k},X_{j}}(t)$ is nondecreasing). We emphasise that \eqref{eq:2.18} holds if for all $n,m \geqslant 1$,
\begin{equation}\label{eq:2.20}
(X_{n},X_{n+m}) \overset{\mathrm{d}}{=} (X_{1},X_{1+m}).
\end{equation}
In addition, one can demonstrate that for any sequence $\{X_{n}, \, n \geqslant 1 \}$ of pairwise PQD random variables satisfying condition \eqref{eq:2.20},
\begin{equation}\label{eq:2.21}
\sum_{j=1}^{\infty} j \int_{j+1}^{\infty} v^{-3} G_{1,j+1}(v) \, \mathrm{d}v < \infty
\end{equation}
and
\begin{equation*}
\sum_{1 \leqslant k < j \leqslant \infty} j^{-2} \int_{-k}^{k} \int_{-j}^{j} \Delta_{X_{1},X_{j}}(x,y) \, \mathrm{d}x \mathrm{d}y < \infty
\end{equation*}
are both equivalent by employing the same proof of \citealt{Chen19} for (positively) associated random sequences (see Appendix of \citealt{Chen19} for details). Obviously, \eqref{eq:2.19} implies \eqref{eq:2.21}.

Similarly, under the assumptions of Theorem~\ref{thr:2}, and the extra conditions $a_{n} = 1$ for all $n$, \eqref{eq:2.18}, condition \eqref{eq:2.9} can be simplified to
\begin{equation*}
\sum_{j=2}^{\infty} \int_{\frac{j^{1/p}}{\Log^{2/p} j}}^{\infty} \frac{G_{X_{1},X_{j}}(t)}{t^{3} \Log^{2} \, t}  \, \mathrm{d}t < \infty.
\end{equation*}
Moreover, in this scenario one can still relax \eqref{eq:2.9} to
\begin{equation*}
\sum_{j=2}^{\infty} \frac{\Log^{2(2 - p)/p} j}{j^{2/p - 1}} G_{X_{1},X_{j}} \left(\frac{j^{1/p}}{\Log^{2/p} j} \right) < \infty.
\end{equation*}
Let us point out that the identical distribution of $\{X_{n}, \, n \geqslant 1 \}$ is not a sufficient condition to obtain \eqref{eq:2.18} as the next example shows.

\begin{example}
Considering the following joint probability function of $(X_{k},X_{j})$, $k < j$,

\begin{center}
\begin{tabular}{|c|c|c|c|} \hline
\backslashbox{\tabular{@{}l@{}}$X_{k}$\endtabular}{$X_{j}$}
    & $0$ & $1$ & \\ \hline
$0$ & $\frac{1}{4} + \frac{1}{2^{k+j}}$ & $\frac{1}{4} - \frac{1}{2^{k+j}}$ & $\frac{1}{2}$ \rule{0pt}{3.0ex} \\[5pt] \hline
$1$ & $\frac{1}{4} - \frac{1}{2^{k+j}}$ & $\frac{1}{4} + \frac{1}{2^{k+j}}$ & $\frac{1}{2}$ \rule{0pt}{3.0ex} \\[5pt] \hline
    & $\frac{1}{2}$ & $\frac{1}{2}$ & \rule{0pt}{3.0ex} \\[5pt] \hline
\end{tabular}
\end{center}
we have $\mathbb{P} \left\{X_{n} = 0 \right\} = 1/2 = \mathbb{P} \left\{X_{n} = 1 \right\}$ for each $n \geqslant 1$ and
\begin{equation*}
\Delta_{X_{k},X_{j}}(x,y) = \frac{1}{2^{k+j}} \neq \frac{1}{2^{1+j}} = \Delta_{X_{1},X_{j}}(x,y), \qquad k > 1
\end{equation*}
for all $x,y < 1$.
\end{example}

\section{Applications}

\subsection{Linear errors-in-variables regression model}

\indent

Consider the simple linear errors-in-variables regression model,
\begin{equation}\label{eq:3.1}
\left\{
\begin{array}{l}
  \eta_{n} = \alpha + \beta x_{n} + \varepsilon_{n} \\[5pt]
  \xi_{n} = x_{n} + \delta_{n}
\end{array}
\right. \qquad (n \geqslant 1)
\end{equation}
where $\alpha, \beta$ are unknown parameters, $x_{1},x_{2},\ldots$ are (non-random) constants and $\left\{\varepsilon_{n}, \, n \geqslant 1 \right\}$, $\left\{\delta_{n}, \, n \geqslant 1 \right\}$ are two sequences of random variables. Recall that the model \eqref{eq:3.1} not only furnishes an approximation to real world situations but also it helps us understand the theoretical underpinnings of methods for other models (see \citealt{Fuller87}). Rewriting \eqref{eq:3.1} as an ordinary regression model having stochastic regressors and errors $\varepsilon_{k} - \beta \delta_{k}$, i.e.
\begin{equation*}
\eta_{n} = \alpha + \beta \xi_{n} + (\varepsilon_{n} - \beta \delta_{n}) \qquad (n \geqslant 1),
\end{equation*}
formally, we can obtain the least-squares estimators of $\beta$ and $\alpha$ as
\begin{equation}\label{eq:3.2}
\widehat{\beta}_{n} := \frac{\sum_{k=1}^{n}\left(\xi_{k} - n^{-1} \sum_{j=1}^{n} \xi_{j} \right) \left(\eta_{k} - n^{-1} \sum_{j=1}^{n} \eta_{j} \right)}{\sum_{k=1}^{n} \left(\xi_{k} - n^{-1} \sum_{k=1}^{n} \xi_{k} \right)^{2}}
\end{equation}
and
\begin{equation}\label{eq:3.3}
\widehat{\alpha}_{n} := \frac{1}{n} \sum_{k=1}^{n} \eta_{k}  - \frac{\widehat{\beta}_{n}}{n} \sum_{k=1}^{n} \xi_{k},
\end{equation}
respectively (see \citealt{Liu05}).

In \citealt{Liu05}, necessary and sufficient conditions were given to ensure the strong consistency of $\widehat{\beta}_{n}$ and $\widehat{\alpha}_{n}$ assuming that $\left\{(\varepsilon_{n},\delta_{n}), \, n  \geqslant 1 \right\}$ is a sequence of independent random vectors, $\left\{\varepsilon_{n}, \, n  \geqslant 1 \right\}$ is a sequence of i.i.d. random variables and $\left\{\delta_{n}, \, n  \geqslant 1 \right\}$ is a sequence of i.i.d. random variables satisfying $\mathbb{E} \, \varepsilon_{1} = \mathbb{E} \, \delta_{1} = 0$, $0 < \mathbb{E} \, \delta_{1}^{2} < \infty$, $0 < \mathbb{E} \, \varepsilon_{1}^{2} < \infty$. Later, admitting that $\left\{(\varepsilon_{n},\delta_{n}), \, n  \geqslant 1 \right\}$ is a sequence of stationary $\alpha$-mixing random vectors, sufficient conditions were given in \citealt{Fan10} to get the strong consistency of $\widehat{\alpha}_{n}$ and $\widehat{\beta}_{n}$. More recently, necessary and sufficient conditions for the strong consistency of these estimators were obtained in \citealt{Hu17} when $\left\{(\varepsilon_{n}, \delta_{n}), \, n  \geqslant 1 \right\}$ is a sequence of identically distributed $\psi$-mixing random vectors.

In order to broaden further the dependence structure of the random components in the model \eqref{eq:3.1}, we shall establish sufficient conditions for the strong consistency of both estimators, $\widehat{\alpha}_{n}$ and $\widehat{\beta}_{n}$, under sequences $\left\{\varepsilon_{n}, \, n  \geqslant 1 \right\}$ and $\left\{\delta_{n}, \, n  \geqslant 1 \right\}$ of pairwise PQD random variables.

Here, $\overline{x}_{n} := \sum_{k=1}^{n} x_{k}/n$ and other similar notations, such as $\overline{\delta}_{n}$ or $\overline{\xi}_{n}$ are defined in the same way.

\begin{theorem}
Suppose that in model \eqref{eq:3.1}, $\{\varepsilon_{n}, \, n \geqslant 1 \}$ is a sequence of pairwise PQD random variables stochastically dominated by a random variable $\varepsilon \in \mathscr{L}_{2}$,
\begin{equation}\label{eq:3.4}
\sum_{1 \leqslant k < j \leqslant \infty} \int_{j}^{\infty} t^{-2} \left[ G_{\varepsilon_{k}^{+},\varepsilon_{j}^{+}}(\sqrt{t}) + G_{\varepsilon_{k}^{-},\varepsilon_{j}^{-}}(\sqrt{t}) \right] \, \mathrm{d}t < \infty
\end{equation}
and $\{\delta_{n}, \, n \geqslant 1 \}$ is a sequence of pairwise PQD random variables stochastically
dominated by a random variable $\delta \in \mathscr{L}_{2}$,
\begin{equation*}
\sum_{1 \leqslant k < j \leqslant \infty} \int_{j}^{\infty} t^{-2} \left[ G_{\delta_{k}^{+},\delta_{j}^{+}}(\sqrt{t}) + G_{\delta_{k}^{-},\delta_{j}^{-}}(\sqrt{t}) \right] \, \mathrm{d}t < \infty.
\end{equation*}
If $n/\sum_{k=1}^{n} (x_{k} - \overline{x}_{n})^{2} = o(1)$ as $n \rightarrow \infty$, then $\widehat{\beta}_{n} \overset{\textnormal{a.s.}}{\longrightarrow} \beta$. Additionally, if $n \, \lvert \overline{x}_{n} \rvert(\lvert\overline{x}_{n} \rvert \vee 1)/\sum_{k=1}^{n} (x_{k} - \overline{x}_{n})^{2} = o(1)$ as $n \rightarrow \infty$ then $\widehat{\alpha}_{n} \overset{\textnormal{a.s.}}{\longrightarrow} \alpha$.
\end{theorem}

\begin{proof}
Supposing $\varepsilon_{n}^{+} := \varepsilon_{n} \vee 0$ and $\varepsilon_{n}^{-} := (- \varepsilon_{n}) \vee 0$, it is straightforward to see that $\left\{(\varepsilon_{n}^{+})^{2}, \, n \geqslant 1 \right\}$ is a sequence of pairwise PQD random variables stochastically dominated by $\varepsilon^{2}$. Since
\begin{equation*}
\begin{split}
G_{(\varepsilon_{k}^{+})^{2},(\varepsilon_{j}^{+})^{2}}(t) &= \int_{-t}^{t} \int_{-t}^{t} \left[\mathbb{P} \left\{(\varepsilon_{k}^{+})^{2} \leqslant x, (\varepsilon_{j}^{+})^{2} \leqslant y \right\} - \mathbb{P} \left\{(\varepsilon_{k}^{+})^{2} \leqslant x \right\} \mathbb{P} \left\{(\varepsilon_{j}^{+})^{2} \leqslant y \right\} \right] \mathrm{d}x \mathrm{d}y \\
& = \int_{0}^{t} \int_{0}^{t} \left[\mathbb{P} \left\{\varepsilon_{k}^{+} \leqslant \sqrt{x}, \varepsilon_{j}^{+} \leqslant \sqrt{y} \right\} - \mathbb{P} \left\{\varepsilon_{k}^{+} \leqslant \sqrt{x} \right\} \mathbb{P} \left\{\varepsilon_{j}^{+} \leqslant \sqrt{y} \right\} \right] \mathrm{d}x \mathrm{d}y \\
& = 4 \int_{0}^{\sqrt{t}} \int_{0}^{\sqrt{t}} uv \left[\mathbb{P} \left\{\varepsilon_{k}^{+} \leqslant u, \varepsilon_{j}^{+} \leqslant v \right\} - \mathbb{P} \left\{\varepsilon_{k}^{+} \leqslant u \right\} \mathbb{P} \left\{\varepsilon_{j}^{+} \leqslant v \right\} \right]\mathrm{d}u \mathrm{d}v \\
& \leqslant 4 t \int_{0}^{\sqrt{t}} \int_{0}^{\sqrt{t}} \left[\mathbb{P} \left\{\varepsilon_{k}^{+} \leqslant u, \varepsilon_{j}^{+} \leqslant v \right\} - \mathbb{P} \left\{\varepsilon_{k}^{+} \leqslant u \right\} \mathbb{P} \left\{\varepsilon_{j}^{+} \leqslant v \right\} \right]\mathrm{d}u \mathrm{d}v \\
& = 4t G_{\varepsilon_{k}^{+},\varepsilon_{j}^{+}}(\sqrt{t})
\end{split}
\end{equation*}
we obtain $\sum_{k=1}^{n} \big[(\varepsilon_{k}^{+})^{2} - \mathbb{E} (\varepsilon_{k}^{+})^{2} \big]/n \overset{\textnormal{a.s.}}{\longrightarrow} 0$ via Theorem~\ref{thr:1}. Thus,
\begin{equation*}
\limsup_{n \rightarrow \infty} \abs{\frac{1}{n} \sum_{k=1}^{n} (\varepsilon_{k}^{+})^{2}} \leqslant \limsup_{n \rightarrow \infty} \left\{\abs{\frac{1}{n} \sum_{k=1}^{n} \left[(\varepsilon_{k}^{+})^{2} - \mathbb{E} (\varepsilon_{k}^{+})^{2} \right]} + \frac{1}{n} \sum_{k=1}^{n} \mathbb{E} (\varepsilon_{k}^{+})^{2} \right\} \leqslant \mathbb{E} \, \varepsilon^{2} \quad \text{a.s.}
\end{equation*}
By analogous reasoning we can conclude ${\displaystyle \limsup_{n \rightarrow \infty}} \abs{\sum_{k=1}^{n} (\varepsilon_{k}^{-})^{2}/n} \leqslant \mathbb{E} \, \varepsilon^{2}$ a.s. and so
\begin{equation}\label{eq:3.5}
\limsup_{n \rightarrow \infty} \abs{\frac{1}{n} \sum_{k=1}^{n} \varepsilon_{k}^{2}} = \limsup_{n \rightarrow \infty} \abs{\frac{1}{n} \sum_{k=1}^{n} \left[(\varepsilon_{k}^{+})^{2} + (\varepsilon_{k}^{-})^{2} \right]} \leqslant 2 \, \mathbb{E} \, \varepsilon^{2} \quad \text{a.s.}
\end{equation}
Similarly,
\begin{equation}\label{eq:3.6}
\limsup_{n \rightarrow \infty} \abs{\frac{1}{n} \sum_{k=1}^{n} \delta_{k}^{2}} \leqslant 2 \, \mathbb{E} \, \delta^{2} \quad \text{a.s.}
\end{equation}
Setting $s_{n} := \sum_{k=1}^{n} (x_{k} - \overline{x}_{n})^{2}$, it follows
\begin{equation}\label{eq:3.7}
\abs{\frac{\sum_{k=1}^{n} (\delta_{k} - \overline{\delta}_{n})\varepsilon_{k}}{s_{n}}} \leqslant \frac{n}{s_{n}} \cdot \frac{1}{n} \sum_{k=1}^{n} \lvert \varepsilon_{k} \delta_{k} \rvert + \frac{n}{s_{n}} \lvert \overline{\delta}_{n} \, \overline{\varepsilon}_{n} \rvert \leqslant \frac{2n}{s_{n}} \left(\frac{1}{n} \sum_{k=1}^{n} \varepsilon_{k}^{2} \right)^{1/2} \left(\frac{1}{n} \sum_{k=1}^{n} \delta_{k}^{2} \right)^{1/2} \overset{\textnormal{a.s.}}{\longrightarrow} 0
\end{equation}
and
\begin{equation}\label{eq:3.8}
\frac{\sum_{k=1}^{n}(\delta_{k} - \overline{\delta}_{n})^{2}}{s_{n}} = \frac{n}{s_{n}} \cdot \frac{1}{n} \sum_{k=1}^{n} \delta_{k}^{2} - \frac{n}{s_{n}} \cdot \overline{\delta}_{n}^{2} \leqslant \frac{n}{s_{n}} \cdot \frac{1}{n} \sum_{k=1}^{n} \delta_{k}^{2} \overset{\textnormal{a.s.}}{\longrightarrow} 0
\end{equation}
from \eqref{eq:3.5} and \eqref{eq:3.6}. Moreover,
\begin{equation*}
\abs{\frac{\sum_{k=1}^{n} (x_{k} - \overline{x}_{n}) \varepsilon_{k}}{s_{n}}} \leqslant \left(\frac{n}{s_{n}} \right)^{1/2} \left(\frac{\sum_{k=1}^{n} \varepsilon_{k}^{2}}{n} \right)^{1/2} \overset{\textnormal{a.s.}}{\longrightarrow} 0
\end{equation*}
and
\begin{equation*}
\abs{\frac{\sum_{k=1}^{n} (x_{k} - \overline{x}_{n}) \delta_{k}}{s_{n}}} \leqslant \left(\frac{n}{s_{n}} \right)^{1/2} \left(\frac{\sum_{k=1}^{n} \delta_{k}^{2}}{n} \right)^{1/2} \overset{\textnormal{a.s.}}{\longrightarrow} 0,
\end{equation*}
yielding
\begin{equation}\label{eq:3.9}
\frac{\sum_{k=1}^{n} (x_{k} - \overline{x}_{n})(\varepsilon_{k} - \beta \delta_{k})}{s_{n}} \overset{\textnormal{a.s.}}{\longrightarrow} 0.
\end{equation}
Thus, \eqref{eq:3.8} entails
\begin{equation*}
\abs{\frac{\sum_{k=1}^{n} (x_{k} - \overline{x}_{n})(\delta_{k} - \overline{\delta}_{n})}{s_{n}}}
\leqslant \left[\frac{\sum_{k=1}^{n} (\delta_{k} - \overline{\delta}_{n})^{2}}{s_{n}} \right]^{1/2}  \overset{\textnormal{a.s.}}{\longrightarrow} 0
\end{equation*}
and also
\begin{equation}\label{eq:3.10}
\frac{\sum_{k=1}^{n}(\xi_{k} - \overline{\xi}_{n})^{2}}{s_{n}} = 1 + \frac{2 \sum_{k=1}^{n}(x_{k} - \overline{x}_{n})(\delta_{k} - \overline{\delta}_{n})}{s_{n}} + \frac{\sum_{k=1}^{n}(\delta_{k} - \overline{\delta}_{n})^{2}}{s_{n}} \overset{\textnormal{a.s.}}{\longrightarrow} 1.
\end{equation}
Since
\begin{align*}
& \widehat{\beta}_{n} - \beta = \\
&\quad = \frac{\sum_{k=1}^{n} (\delta_{k} - \overline{\delta}_{n})\varepsilon_{k} + \sum_{k=1}^{n} (x_{k} - \overline{x}_{n})(\varepsilon_{k} - \beta \overline{\delta}_{k}) - \beta \sum_{k=1}^{n}(\delta_{k} - \overline{\delta}_{n})^{2}}{\sum_{k=1}^{n} (\xi_{k} - \overline{\xi}_{n})^{2}} \\
&\quad = \frac{s_{n}}{\sum_{k=1}^{n} (\xi_{k} - \overline{\xi}_{n})^{2}} \left[\frac{\sum_{k=1}^{n} (\delta_{k} - \overline{\delta}_{n})\varepsilon_{k}}{s_{n}} + \frac{\sum_{k=1}^{n} (x_{k} - \overline{x}_{n})(\varepsilon_{k} - \beta \overline{\delta}_{k})}{s_{n}} - \beta \frac{\sum_{k=1}^{n}(\delta_{k} - \overline{\delta}_{n})^{2}}{s_{n}} \right]
\end{align*}
we obtain $\widehat{\beta}_{n} \overset{\textnormal{a.s.}}{\longrightarrow} \beta$ from \eqref{eq:3.7}, \eqref{eq:3.8}, \eqref{eq:3.9} and \eqref{eq:3.10}. On the other hand,
\begin{equation*}
\widehat{\alpha}_{n} - \alpha = (\beta - \widehat{\beta}_{n}) \overline{x}_{n} + (\beta - \widehat{\beta}_{n}) \overline{\delta}_{n} - \beta \overline{\delta}_{n} + \overline{\varepsilon}_{n}.
\end{equation*}
According to Theorem~\ref{thr:1}, $\overline{\varepsilon}_{n} \overset{\textnormal{a.s.}}{\longrightarrow} 0$ and $\overline{\delta}_{n} \overset{\textnormal{a.s.}}{\longrightarrow} 0$. Hence, it suffices to prove
\begin{equation}\label{eq:3.11}
(\beta - \widehat{\beta}_{n}) \overline{x}_{n} \overset{\textnormal{a.s.}}{\longrightarrow} 0.
\end{equation}
We have
\begin{equation}\label{eq:3.12}
\begin{split}
& \abs{\frac{\overline{x}_{n}}{s_{n}} \sum_{k=1}^{n} (\delta_{k} - \overline{\delta}_{n})\varepsilon_{k}} \leqslant \frac{n \lvert \overline{x}_{n} \rvert}{s_{n}} \cdot \frac{1}{n} \sum_{k=1}^{n} \lvert \varepsilon_{k} \delta_{k} \rvert + \frac{n \lvert \overline{x}_{n} \rvert}{s_{n}} \cdot  \lvert \overline{\delta}_{n} \, \overline{\varepsilon}_{n} \rvert \leqslant \\
&\quad \leqslant \frac{2 n \lvert \overline{x}_{n} \rvert}{s_{n}} \cdot \left(\frac{1}{n} \sum_{k=1}^{n} \varepsilon_{k}^{2} \right)^{1/2} \left(\frac{1}{n} \sum_{k=1}^{n} \delta_{k}^{2} \right)^{1/2} \overset{\textnormal{a.s.}}{\longrightarrow} 0
\end{split}
\end{equation}
and
\begin{equation}\label{eq:3.13}
\begin{split}
&\abs{\frac{\overline{x}_{n}}{s_{n}} \sum_{k=1}^{n}(\delta_{k} - \overline{\delta}_{n})^{2}} = \frac{n \lvert \overline{x}_{n} \rvert}{s_{n}} \cdot \frac{1}{n} \sum_{k=1}^{n}(\delta_{k} - \overline{\delta}_{n})^{2} = \\
&\qquad = \frac{n \lvert \overline{x}_{n} \rvert}{s_{n}} \left(\frac{1}{n} \sum_{k=1}^{n} \delta_{k}^{2} - \overline{\delta}_{n}^{2} \right) \leqslant \frac{n \lvert \overline{x}_{n} \rvert}{s_{n}} \cdot \frac{1}{n} \sum_{k=1}^{n} \delta_{k}^{2} \overset{\textnormal{a.s.}}{\longrightarrow} 0.
\end{split}
\end{equation}
Moreover,
\begin{align*}
& \abs{\frac{\overline{x}_{n} \sum_{k=1}^{n} (x_{k} - \overline{x}_{n})\varepsilon_{k}}{s_{n}}} \leqslant \frac{\lvert \overline{x}_{n} \rvert}{s_{n}} \cdot \left[\sum_{k=1}^{n} (x_{k} - \overline{x}_{n})^{2} \right]^{1/2} \left(\sum_{k=1}^{n} \varepsilon_{k}^{2} \right)^{1/2} = \\
& \qquad = \left(\frac{n \, \overline{x}_{n}^{2}}{s_{n}} \right)^{1/2} \cdot \left(\frac{1}{n} \sum_{k=1}^{n} \varepsilon_{k}^{2} \right)^{1/2} \overset{\textnormal{a.s.}}{\longrightarrow} 0
\end{align*}
and
\begin{align*}
& \abs{\frac{\overline{x}_{n} \sum_{k=1}^{n} (x_{k} - \overline{x}_{n})\delta_{k}}{s_{n}}} \leqslant \frac{\lvert \overline{x}_{n} \rvert}{s_{n}} \cdot \left[\sum_{k=1}^{n} (x_{k} - \overline{x}_{n})^{2} \right]^{1/2} \left(\sum_{k=1}^{n} \delta_{k}^{2} \right)^{1/2} = \\
& \qquad = \left(\frac{n \, \overline{x}_{n}^{2}}{s_{n}} \right)^{1/2} \cdot \left(\frac{1}{n} \sum_{k=1}^{n} \delta_{k}^{2} \right)^{1/2} \overset{\textnormal{a.s.}}{\longrightarrow} 0
\end{align*}
imply
\begin{equation}\label{eq:3.14}
\frac{\overline{x}_{n} \sum_{k=1}^{n} (x_{k} - \overline{x}_{n})(\varepsilon_{k} - \beta \delta_{k})}{s_{n}} \overset{\textnormal{a.s.}}{\longrightarrow} 0.
\end{equation}
Thus,
\begin{align*}
&\overline{x}_{n} (\beta - \widehat{\beta}_{n}) = - \frac{s_{n}}{\sum_{k=1}^{n}(\xi_{k} - \overline{\xi}_{n})^{2}} \cdot \\
&\qquad \left[\frac{\overline{x}_{n} \sum_{k=1}^{n} (\delta_{k} - \overline{\delta}_{n})\varepsilon_{k}}{s_{n}} + \frac{\overline{x}_{n} \sum_{k=1}^{n} (x_{k} - \overline{x}_{n})(\varepsilon_{k} - \beta \overline{\delta}_{k})}{s_{n}} - \beta \frac{\overline{x}_{n} \sum_{k=1}^{n}(\delta_{k} - \overline{\delta}_{n})^{2}}{s_{n}} \right]
\end{align*}
and \eqref{eq:3.11} holds from \eqref{eq:3.10}, \eqref{eq:3.12}, \eqref{eq:3.13} and \eqref{eq:3.14}. The proof is complete.
\end{proof}

\begin{remark}
Let us note that if $\overline{x}_{n}$ is bounded then condition $n \lvert \overline{x}_{n} \rvert (\lvert \overline{x}_{n} \rvert \vee 1)/\sum_{k=1}^{n} (x_{k} - \overline{x}_{n})^{2} = o(1)$, $n \rightarrow \infty$ can be dropped (that is, $n/\sum_{k=1}^{n} (x_{k} - \overline{x}_{n})^{2} = o(1)$, $n \rightarrow \infty$ is sufficient to obtain strong consistency of both estimators $\widehat{\alpha}_{n}$ and $\widehat{\beta}_{n}$).
\end{remark}

\subsection{Multiple regression model}

\indent

Consider the multiple regression model
\begin{equation} \label{eq:3.15}
\mathbf{y}_{n} = \mathbf{X}_{n} \boldsymbol{\beta} + \boldsymbol{\varepsilon}_{n}
\end{equation}
where $\mathbf{X}_{n} = \big(x_{ij} \big)_{1 \leqslant i \leqslant n, 1 \leqslant j \leqslant p}$ is a known $n \times p$ matrix of rank $p$, $\boldsymbol{\beta} = (\beta_{1}, \ldots, \beta_{p})'$ is the $p$-dimensional parameter vector, $\boldsymbol{\varepsilon}_{n} = (\varepsilon_{1}, \ldots, \varepsilon_{n})'$ the $n$-dimensional error vector and $\mathbf{y}_{n} = (y_{1}, \ldots, y_{n})'$ the $n$-dimensional observation vector with prime denoting transpose. For $n \geqslant p$,
\begin{equation*}
\boldsymbol{\widehat{\beta}}_{n} = \boldsymbol{\beta} + (\mathbf{X}_{n}' \mathbf{X}_{n})^{-1} \mathbf{X}_{n}' \boldsymbol{\varepsilon}_{n}
\end{equation*}
is the least-squares estimate of $\boldsymbol{\beta}$.

\subsubsection{Non-stochastic regressors}

\indent

The strong consistency of the least squares estimates in multiple regression models having non-stochastic regressors was studied in the past by many authors (see, \citealt{Drygas76}, \citealt{Guijing81} or \citealt{Lai79}, among others). In the following, the strong consistency for least-squares estimators of unknown parameter vector is given. It extends Theorem 1 of \citealt{Lita13} to sequences $\{\varepsilon_{n}, \, n \geqslant 1 \}$ of pairwise PQD random variables.

\begin{theorem}\label{thr:4}
Suppose that in model \eqref{eq:3.15}, $\{\varepsilon_{n}, \, n \geqslant 1 \}$ is a sequence of identically distributed pairwise PQD random variables such that $\varepsilon_{1} \in \mathscr{L}_{r}$ for some  $1 \leqslant r < 2$ and $\mathbb{E} \, \varepsilon_{1} = 0$. If $\mathbf{X}_{n}' \mathbf{X}_{n}$ is non-singular for some $n \geqslant n_{0}$, the design levels $\{x_{ij}, \, 1 \leqslant j \leqslant p, i \geqslant 1 \}$ satisfy $\sup_{n \geqslant 1} \sum_{k=1}^{n} x_{kj}^{2}/n < \infty$ for every $j$,
\begin{itemize}
\item[\textnormal{(i)}] $\left[(\mathbf{X}_{n}' \mathbf{X}_{n})^{-1} \right]_{jj} = O\left(n^{-1} \right)$ as $n \rightarrow \infty$ for all $j$ and
    \begin{equation*}
    \sum_{1 \leqslant k < \ell \leqslant \infty} \lvert x_{ki} x_{\ell j} \rvert \int_{\ell}^{\infty} t^{-3} G_{\varepsilon_{k},\varepsilon_{\ell}}(t) \, \mathrm{d}t < \infty \qquad (i,j=1,\ldots,p)
    \end{equation*}
    when $r = 1$,
\end{itemize}
or
\begin{itemize}
\item[\textnormal{(ii)}] $\left[(\mathbf{X}_{n}' \mathbf{X}_{n})^{-1} \right]_{jj} = O \left(n^{-1/r} \, \Log^{-2(r - 1)/r} n \right)$ as $n \rightarrow \infty$ for all $j$ and
    \begin{equation*}
    \sum_{1 \leqslant k < \ell \leqslant \infty} \lvert x_{ki} x_{\ell j} \rvert \int_{\frac{\ell^{1/r}}{\Log^{2/r} \ell}}^{\infty} \frac{G_{\varepsilon_{k},\varepsilon_{\ell}}(t)}{t^{3} \Log^{2} \, t} \, \mathrm{d}t < \infty \qquad (i,j = 1,\ldots,p)
    \end{equation*}
    whenever $1 < r < 2$,
\end{itemize}
then $\boldsymbol{\widehat{\beta}}_{n} \overset{\textnormal{a.s.}}{\longrightarrow} \boldsymbol{\beta}$.
\end{theorem}

\begin{proof}
From the expression of $\boldsymbol{\widehat{\beta}}_{n}$, it follows that the strong consistency of the least-squares estimate is equivalent to
\begin{equation*}
(\mathbf{X}_{n}' \mathbf{X}_{n})^{-1} \sum_{k=1}^{n} \mathbf{x}_{k} \varepsilon_{k} \overset{\textnormal{a.s.}}{\longrightarrow} \mathbf{0}
\end{equation*}
where $\mathbf{x}_{k} = (x_{k1},\ldots,x_{kp})'$. Since $(\mathbf{X}_{n}' \mathbf{X}_{n})^{-1}$, $n \geqslant n_{0}$ is symmetric positive-definite, we have
\begin{equation*}
\abs{\left[(\mathbf{X}_{n}' \mathbf{X}_{n})^{-1} \right]_{ij}} \leqslant \left[(\mathbf{X}_{n}' \mathbf{X}_{n})^{-1} \right]_{ii}^{1/2} \left[(\mathbf{X}_{n}' \mathbf{X}_{n})^{-1} \right]_{jj}^{1/2}
\end{equation*}
(see \citealt{Harville97}, page $280$). Hence,
\begin{align*}
& \abs{\left[(\mathbf{X}_{n}' \mathbf{X}_{n})^{-1} \right]_{ij} \sum_{k=1}^{n} x_{kj} \varepsilon_{k}} \leqslant \\
& \qquad \left[(\mathbf{X}_{n}' \mathbf{X}_{n})^{-1} \right]_{ii}^{1/2} \left[(\mathbf{X}_{n}' \mathbf{X}_{n})^{-1} \right]_{jj}^{1/2} \abs{\sum_{k=1}^{n} x_{kj} \varepsilon_{k}} \leqslant C \abs{\frac{1}{n} \sum_{k=1}^{n} x_{kj} \varepsilon_{k}} \overset{\textnormal{a.s.}}{\longrightarrow} 0
\end{align*}
when $r = 1$ by Theorem~\ref{thr:1}; from Theorem~\ref{thr:2}, we get
\begin{align*}
& \abs{\left[(\mathbf{X}_{n}' \mathbf{X}_{n})^{-1} \right]_{ij} \sum_{k=1}^{n} x_{kj} \varepsilon_{k}} \leqslant \\
& \qquad \left[(\mathbf{X}_{n}' \mathbf{X}_{n})^{-1} \right]_{ii}^{1/2} \left[(\mathbf{X}_{n}' \mathbf{X}_{n})^{-1} \right]_{jj}^{1/2} \abs{\sum_{k=1}^{n} x_{kj} \varepsilon_{k}} \leqslant C \abs{\frac{1}{n^{1/r} \, \Log^{2(r - 1)/r} n} \sum_{k=1}^{n} x_{kj} \varepsilon_{k}} \overset{\textnormal{a.s.}}{\longrightarrow} 0
\end{align*}
whenever $1 < r < 2$ establishing the thesis.
\end{proof}

\subsubsection{Stochastic regressors}

\indent

In model \eqref{eq:3.15}, let us assume that the design levels $\{x_{ij}, 1 \leqslant j \leqslant p, i \geqslant 1 \}$ are random variables. If the errors $\varepsilon_{1}, \varepsilon_{2},\ldots$ are pairwise PQD and identically distributed random variables then we can use Theorem~\ref{thr:1} to prove the strong consistency of $\boldsymbol{\widehat{\beta}}_{n}$.

In what follows, we shall define $\rho(\mathbf{A}) = \sup \left\{\abs{\lambda}\colon \lambda \in \mathrm{Spec}(\mathbf{A}) \right\}$ where $\mathrm{Spec}(\mathbf{A})$ is the spectrum of the matrix $\mathbf{A} = \big(a_{ij} \big)_{1 \leqslant i,j \leqslant p}$. The column space of the matrix $\mathbf{M} = \big(m_{ij} \big)_{1 \leqslant i \leqslant n, 1 \leqslant j \leqslant p}$ will be indicated by $\mathrm{Col}(\mathbf{M})$. Given a $n$-dimensional vector $\mathbf{a}$ we shall use $\norm{\mathbf{a}}$ to denote the \emph{Euclidean} vector norm, that is, $\norm{\mathbf{a}} = \sqrt{\mathbf{a}' \mathbf{a}}$.

\begin{theorem}\label{thr:5}
Suppose that in model \eqref{eq:3.15}, $\{\varepsilon_{n}, \, n \geqslant 1 \}$ is a sequence of pairwise PQD random variables stochastically dominated by a random variable $\varepsilon \in \mathscr{L}_{2}$ satisfying \eqref{eq:3.4}. If $\{x_{ij} \}$ $(i=1,2,\ldots; j=1,\ldots,p)$ is an arbitrary double array of random variables such that $\mathbf{X}_{n}' \mathbf{X}_{n}$ is non-singular a.s. for some $n \geqslant p$ and $n \rho \left((\mathbf{X}_{n}' \mathbf{X}_{n})^{-1} \right) \overset{\textnormal{a.s.}}{\longrightarrow} 0$, then $\boldsymbol{\widehat{\beta}}_{n} \overset{\textnormal{a.s.}}{\longrightarrow} \boldsymbol{\beta}$.
\end{theorem}

\begin{proof}
From Proposition 1 of \citealt{Lita14}, we have
\begin{equation}\label{eq:3.16}
\lvert \! \lvert \boldsymbol{\widehat{\beta}}_{n} - \boldsymbol{\beta} \rvert \! \rvert^{2} \leqslant \rho \left((\mathbf{X}_{n}' \mathbf{X}_{n})^{-1} \right) \lvert \! \lvert \mathbf{P}_{\mathrm{Col}(\mathbf{X}_{n})} \boldsymbol{\varepsilon}_{n} \rvert \! \rvert^{2} \quad \text{a.s.}
\end{equation}
where $\mathbf{P}_{\mathrm{Col}(\mathbf{X}_{n})} \boldsymbol{\varepsilon}_{n}$ is the orthogonal projection of $\boldsymbol{\varepsilon}_{n}$ on $\mathrm{Col}(\mathbf{X}_{n})$. Using Gram-Schmidt process we can construct an orthonormal basis $\left\{\mathbf{w}_{n,1}, \ldots, \mathbf{w}_{n,p} \right\}$ of $\mathrm{Col}(\mathbf{X}_{n})$ such that
\begin{equation}\label{eq:3.17}
\lvert \! \lvert \mathbf{P}_{\mathrm{Col}(\mathbf{X}_{n})} \boldsymbol{\varepsilon}_{n} \rvert \! \rvert^{2} = \langle \mathbf{w}_{n,1}, \boldsymbol{\varepsilon}_{n} \rangle^{2} + \ldots + \langle \mathbf{w}_{n,p}, \boldsymbol{\varepsilon}_{n} \rangle^{2}
\end{equation}
where $\langle \; \cdot \; , \; \cdot \; \rangle$ denotes the usual inner product in $\mathbb{R}^{n}$. From Cauchy-Schwarz inequality we have, for each $j=1,\ldots,p$,
\begin{equation}\label{eq:3.18}
\rho \left((\mathbf{X}_{n}' \mathbf{X}_{n})^{-1} \right) \langle \mathbf{w}_{n,j}, \boldsymbol{\varepsilon}_{n} \rangle^{2} \leqslant \rho \left((\mathbf{X}_{n}' \mathbf{X}_{n})^{-1} \right) \lvert \! \lvert \mathbf{w}_{n,j} \rvert \! \rvert^{2} \lvert \! \lvert \boldsymbol{\varepsilon}_{n} \rvert \! \rvert^{2} = n \rho \left((\mathbf{X}_{n}' \mathbf{X}_{n})^{-1} \right) \cdot \frac{1}{n} \sum_{k=1}^{n} \varepsilon_{k}^{2} \overset{\textnormal{a.s.}}{\longrightarrow} 0
\end{equation}
via \eqref{eq:3.5}. By \eqref{eq:3.16}, \eqref{eq:3.17} and \eqref{eq:3.18}, it follows
\begin{equation*}
\norm{\mathbf{b}_{n} - \boldsymbol{\beta}}^{2} \leqslant \rho \left((\mathbf{X}_{n}' \mathbf{X}_{n})^{-1} \right) \left[\langle \mathbf{w}_{n,1}, \boldsymbol{\varepsilon}_{n} \rangle^{2} + \ldots + \langle \mathbf{w}_{n,p}, \boldsymbol{\varepsilon}_{n} \rangle^{2} \right] \overset{\textnormal{a.s.}}{\longrightarrow} 0.
\end{equation*}
The proof is complete.
\end{proof}

\subsection{Simple ridge regression model}

\indent

In model \eqref{eq:3.15}, suppose $p=1$ and the ridge estimator
\begin{equation*}
\widehat{\gamma}_{n} = \left(\sum_{j=1}^{n} x_{j}^{2} + \kappa \right)^{-1} \mathbf{x}_{n}' \mathbf{y}_{n}
\end{equation*}
where $\kappa = \widehat{\sigma}_{n}^{2}/\widehat{\beta}_{n}^{2}$, $\mathbf{x}_{n} = (x_{1},\ldots,x_{n})'$ and $\widehat{\sigma}_{n}^{2} = (\mathbf{y}_{n} - \mathbf{x}_{n} \widehat{\beta}_{n})' (\mathbf{y}_{n} - \mathbf{x}_{n} \widehat{\beta}_{n})/(n - 1)$ (see \citealt{Grob03}, page $9$).

\begin{theorem}\label{thr:6}
Suppose model \eqref{eq:3.15} with $p=1$ and $\{\varepsilon_{n}, \, n \geqslant 1 \}$ a sequence of pairwise PQD random variables stochastically dominated by a random variable $\varepsilon \in \mathscr{L}_{2}$ satisfying \eqref{eq:3.4}. If $\{x_{n}, \, n \geqslant 1 \}$ is an arbitrary sequence of random variables such that $\sum_{j=1}^{n} x_{j}^{2} \neq 0$ a.s. for some $n \geqslant 1$ and $n/\sum_{j=1}^{n} x_{j}^{2} \overset{\textnormal{a.s.}}{\longrightarrow} 0$, then $\widehat{\gamma}_{n} \overset{\textnormal{a.s.}}{\longrightarrow} \beta$.
\end{theorem}

\begin{proof}
We have
\begin{equation*}
\left(\sum_{j=1}^{n} x_{j}^{2} + \kappa \, \right)^{-1} = \left(1 - \frac{\kappa/\sum_{j=1}^{n} x_{j}^{2}}{1 + \kappa/\sum_{j=1}^{n} x_{j}^{2}} \right) \left(\sum_{j=1}^{n} x_{j}^{2} \right)^{-1}.
\end{equation*}
which yields
\begin{equation*}
\widehat{\gamma}_{n} = \left(1 - \frac{\kappa/\sum_{j=1}^{n} x_{j}^{2}}{1 + \kappa/\sum_{j=1}^{n} x_{j}^{2}} \right)  \beta + \left(1 - \frac{\kappa/\sum_{j=1}^{n} x_{j}^{2}}{1 + \kappa/\sum_{j=1}^{n} x_{j}^{2}} \right) \cdot \frac{\sum_{j=1}^{n} x_{j} \varepsilon_{j}}{\sum_{j=1}^{n} x_{j}^{2}}.
\end{equation*}
Since,
\begin{equation*}
\widehat{\sigma}_{n}^{2} = \frac{\norm{[\mathbf{I}_{n} - \mathbf{x}_{n} (\mathbf{x}_{n}' \mathbf{x}_{n})^{-1} \mathbf{x}_{n}'] \boldsymbol{\varepsilon}_{n}}^{2}}{n - 1} \leqslant \frac{\norm{\boldsymbol{\varepsilon}_{n}}^{2}}{n - 1}
\end{equation*}
it follows
\begin{equation*}
\frac{\widehat{\sigma}_{n}^{2}}{\sum_{j=1}^{n} x_{j}^{2}} \leqslant \frac{\sum_{j=1}^{n} \varepsilon_{j}^{2}}{(n - 1) \sum_{j=1}^{n} x_{j}^{2}} = \frac{n}{n - 1} \cdot \frac{\sum_{j=1}^{n} \varepsilon_{j}^{2}}{n^{2}} \cdot \frac{n}{\sum_{j=1}^{n} x_{j}^{2}}
\end{equation*}
Recall that $\sum_{j=1}^{n} \varepsilon_{j}^{2}/n^{2} \overset{\textnormal{a.s.}}{\longrightarrow} 0$ via Kronecker's lemma provided that $\sum_{n=1}^{\infty} \mathbb{E} \, \varepsilon_{n}^{2}/n^{2} \leqslant C \, \mathbb{E} \, \varepsilon^{2} < \infty$. From Theorem~\ref{thr:5}, we obtain strong consistency of $\widehat{\beta}_{n}$ i.e. $\widehat{\beta}_{n} \overset{\textnormal{a.s.}}{\longrightarrow} \beta$. Thus,
\begin{equation}\label{eq:3.19}
\frac{\kappa}{\sum_{j=1}^{n} x_{j}^{2}} = \frac{1}{\widehat{\beta}_{n}^{2}} \cdot \frac{\widehat{\sigma}_{n}^{2}}{\sum_{j=1}^{n} x_{j}^{2}} \overset{\textnormal{a.s.}}{\longrightarrow} 0
\end{equation}
provided that $\beta \neq 0$; obviously, for $\beta = 0$ one have
\begin{equation}\label{eq:3.20}
\lvert \widehat{\gamma}_{n} \rvert = \left(\sum_{j=1}^{n} x_{j}^{2} + \kappa \right)^{-1} \abs{\sum_{j=1}^{n} x_{j} \varepsilon_{j}} \leqslant\frac{\big\lvert \sum_{j=1}^{n} x_{j} \varepsilon_{j} \big\rvert}{\sum_{j=1}^{n} x_{j}^{2}}.
\end{equation}
On the other hand,
\begin{equation}\label{eq:3.21}
\left(\frac{\sum_{j=1}^{n} x_{j} \varepsilon_{j}}{\sum_{j=1}^{n} x_{j}^{2}} \right)^{2} \leqslant \frac{\sum_{j=1}^{n} \varepsilon_{j}^{2}}{\sum_{j=1}^{n} x_{j}^{2}} = \frac{\sum_{j=1}^{n} \varepsilon_{j}^{2}}{n} \cdot \frac{n}{\sum_{j=1}^{n} x_{j}^{2}}
\end{equation}
and \eqref{eq:3.5} ensures $\sum_{j=1}^{n} x_{j} \varepsilon_{j}/\sum_{j=1}^{n} x_{j}^{2} \overset{\textnormal{a.s.}}{\longrightarrow} 0$ because $n/\sum_{j=1}^{n} x_{j}^{2} \overset{\textnormal{a.s.}}{\longrightarrow} 0$. Hence, \eqref{eq:3.19} (or \eqref{eq:3.20}) and \eqref{eq:3.21} lead to
$\widehat{\gamma}_{n} \overset{\textnormal{a.s.}}{\longrightarrow} \beta$ and
the strong consistency of the ridge estimator $\widehat{\gamma}_{n}$ is established.
\end{proof}

\begin{remark}
Similarly, under the assumptions of Theorem~\ref{thr:6} we can conclude also the strong consistency of the shrinkage estimator $\widehat{\theta}_{n} = \widehat{\beta}_{n}/(1 + \varrho)$, where $\varrho = \left(\sum_{j=1}^{n} x_{j}^{2} \right)^{-1} \widehat{\sigma}_{n}^{2}/\widehat{\beta}_{n}^{2}$ (see \citealt{Grob03}, page $9$).
\end{remark}

\begin{acknowledgements}
The author is grateful to Editor-in-Chief Prof. Narayanaswamy Balakrishnan and Editors who handled the submitted manuscript. The author also wishes to express his gratitude to both Referees for the very careful reading of the manuscript, which resulted in elimination of several typos and improved readability of this final version.

This work is a contribution to the Project UIDB/04035/2020, funded by FCT - Funda\c{c}\~{a}o para a Ci\^{e}ncia e a Tecnologia, Portugal.
\end{acknowledgements}

\end{document}